\documentclass[11pt,letterpaper]{article}
\usepackage[margin=1in]{geometry}
\usepackage{setspace}
\usepackage{cite}
\usepackage{amsmath,amssymb,amsthm,bm,bbm}
\usepackage{mathtools}
\usepackage{graphicx}
\usepackage{microtype}
\usepackage{xcolor}
\usepackage{dsfont}
\usepackage[hidelinks,bookmarks=false]{hyperref}
\theoremstyle{plain}
\newtheorem{theorem}{Theorem}[section]
\newtheorem{lemma}[theorem]{Lemma}
\newtheorem{proposition}[theorem]{Proposition}
\newtheorem{corollary}[theorem]{Corollary}
\newtheorem{fact}[theorem]{Fact}

\theoremstyle{definition}
\newtheorem{definition}[theorem]{Definition}
\newtheorem{assumption}[theorem]{Assumption}

\theoremstyle{remark}
\newtheorem{remark}[theorem]{Remark}

\numberwithin{equation}{section}

\newcommand{\R}{\mathbb{R}}
\newcommand{\E}{\mathbb{E}}
\newcommand{\Prob}{\mathbb{P}}
\newcommand{\calX}{\mathcal{X}}
\newcommand{\calU}{\mathcal{U}}
\newcommand{\calY}{\mathcal{Y}}

\newcommand{\calW}{\mathcal{W}}
\newcommand{\calP}{\mathcal{P}}
\newcommand{\calM}{\mathcal{M}}
\newcommand{\calB}{\mathcal{B}}
\newcommand{\1}{\mathbbm{1}}
\newcommand{\pushfwd}[1]{#1_{\#}}
\newcommand{\norm}[1]{\left\lVert #1 \right\rVert}

\newcommand{\id}{\mathrm{id}}

\DeclareMathOperator{\argmin}{arg\,min}
\DeclareMathOperator{\col}{col}

\newcommand{\nExpFiveCap}{0.10}  
\newcommand{\nExpFiveCapCount}{7}  
\newcommand{\nExpFiveControlScale}{2}  
\newcommand{\nExpFiveDualResidual}{8.14\times 10^{-8}}  
\newcommand{\nExpFiveDynamicsDivisor}{2}  
\newcommand{\nExpFiveElimBlock}{6}  
\newcommand{\nExpFiveElimLocalBlock}{3}  
\newcommand{\nExpFiveElimMoments}{15}  
\newcommand{\nExpFiveElimVariables}{2}  
\newcommand{\nExpFiveEndpointBound}{0.5}  
\newcommand{\nExpFiveEndpointSquareBound}{0.25}  
\newcommand{\nExpFiveEnergyCoeff}{4}  
\newcommand{\nExpFiveEta}{8.19296}  
\newcommand{\nExpFiveFirstTime}{1}  
\newcommand{\nExpFiveFullBlock}{28}  
\newcommand{\nExpFiveFullMoments}{210}  
\newcommand{\nExpFiveFullVariables}{6}  
\newcommand{\nExpFiveGap}{2.22\times 10^{-8}}  
\newcommand{\nExpFiveGraphBlock}{10}  
\newcommand{\nExpFiveGraphCompressed}{9}  
\newcommand{\nExpFiveGraphLocalBlock}{4}  
\newcommand{\nExpFiveGraphMoments}{35}  
\newcommand{\nExpFiveGraphResidual}{9.61\times 10^{-11}}  
\newcommand{\nExpFiveGraphVariables}{3}  
\newcommand{\nExpFiveHorizon}{3}  
\newcommand{\nExpFiveInitialState}{0}  
\newcommand{\nExpFiveLastCap}{0.24}  
\newcommand{\nExpFiveLastTime}{3}  
\newcommand{\nExpFiveLower}{0.8097617}  
\newcommand{\nExpFiveMarkovCross}{0}  
\newcommand{\nExpFiveNlpCost}{0.8097617}  
\newcommand{\nExpFiveNlpStarts}{16}  
\newcommand{\nExpFiveOrder}{2}  
\newcommand{\nExpFiveReset}{0}  
\newcommand{\nExpFiveResetTime}{2}  
\newcommand{\nExpFiveSdpEta}{8.19298}  
\newcommand{\nExpFiveSignalHi}{1}  
\newcommand{\nExpFiveSignalLo}{-1}  
\newcommand{\nExpFiveSolverVersion}{0.11.1}  
\newcommand{\nExpFiveTableRows}{0.01 & 0.080099 & 0.080100 & 0.080100 \\
0.03 & 0.240893 & 0.240894 & 0.240893 \\
0.06 & 0.483547 & 0.483548 & 0.483548 \\
0.10 & 0.809761 & 0.809762 & 0.809762 \\
0.15 & 1.221713 & 1.221714 & 1.221714 \\
0.20 & 1.638177 & 1.638178 & 1.638177 \\
0.24 & 1.974684 & 1.974685 & 1.974684 \\}  
\newcommand{\nExpFiveTolerance}{1\times 10^{-8}}  
\newcommand{\nExpFiveUpper}{0.8097618}  
\newcommand{\nExpFourAffineCapViolation}{8.01\times 10^{-12}}  
\newcommand{\nExpFourAffineCostDiff}{3.26\times 10^{-11}}  
\newcommand{\nExpFourCap}{0.000408881}  
\newcommand{\nExpFourCapFraction}{0.25}  
\newcommand{\nExpFourCapScale}{1.63552\times 10^{-3}}  
\newcommand{\nExpFourConditionalNorm}{1.54\times 10^{-11}}  
\newcommand{\nExpFourCovIncreasePct}{28.2}  
\newcommand{\nExpFourFirstCapped}{3}  
\newcommand{\nExpFourFuture}{4}  
\newcommand{\nExpFourGainPast}{\begin{bmatrix}1.3156 & 1.5129 & -0.8922 & -2.2594\\0.5136 & 0.5906 & -0.1077 & -0.6053\\0.2930 & 0.3370 & -0.0243 & -0.3027\\0 & 0 & 0 & 0\end{bmatrix}}  
\newcommand{\nExpFourInitialHiddenVar}{0.01}  
\newcommand{\nExpFourInitialOutputVar}{0.01}  
\newcommand{\nExpFourInputPenalty}{1}  
\newcommand{\nExpFourJcon}{1.6000189}  
\newcommand{\nExpFourJmodel}{1.6000189}  
\newcommand{\nExpFourLastStep}{5}  
\newcommand{\nExpFourMaxCostGap}{3.04\times 10^{-11}}  
\newcommand{\nExpFourMaxGainGap}{4.45\times 10^{-6}}  
\newcommand{\nExpFourMeanInputsRow}{\begin{bmatrix}0.60406 & 0.19072 & 0.04173 & 0\end{bmatrix}}  
\newcommand{\nExpFourNominalCost}{1.5824848}  
\newcommand{\nExpFourPastInputVar}{0.01}  
\newcommand{\nExpFourReference}{1}  
\newcommand{\nExpFourSensitivityCount}{2}  
\newcommand{\nExpFourSensitivityRows}{0.10, 0.10 & 0.17534127 & 0.17534127 \\
0.01, 0.04 & 0.01962042 & 0.01962042 \\}  
\newcommand{\nExpFourSolverVersion}{0.11.1}  
\newcommand{\nExpFourSweepRows}{0.00 & $0$ & 0.02105600 & 0.02105600 \\
0.05 & $8.18\times 10^{-5}$ & 0.01910813 & 0.01910813 \\
0.10 & $1.64\times 10^{-4}$ & 0.01850458 & 0.01850458 \\
0.25 & $4.09\times 10^{-4}$ & 0.01753413 & 0.01753413 \\
0.50 & $8.18\times 10^{-4}$ & 0.01680158 & 0.01680158 \\
1.00 & $1.64\times 10^{-3}$ & 0.01641889 & 0.01641889 \\
2.00 & $3.27\times 10^{-3}$ & 0.01641889 & 0.01641889 \\}  
\newcommand{\nExpFourTini}{2}  
\newcommand{\nExpFourTolerance}{1\times 10^{-10}}  
\newcommand{\nExpOneAreaOne}{3.3324}  
\newcommand{\nExpOneAreaThree}{2.7173}  
\newcommand{\nExpOneAreaTwo}{2.7173}  
\newcommand{\nExpOneBoxBound}{1}  
\newcommand{\nExpOneFallbackTol}{1\times 10^{-7}}  
\newcommand{\nExpOneFallbackVersion}{3.2.7}  
\newcommand{\nExpOneGraphResidual}{6.76\times 10^{-10}}  
\newcommand{\nExpOneNangles}{144}  
\newcommand{\nExpOneNsamples}{500}  
\newcommand{\nExpOneOrderOne}{1}  
\newcommand{\nExpOneOrderThree}{3}  
\newcommand{\nExpOneOrderTwo}{2}  
\newcommand{\nExpOneReferenceArea}{2.7170}  
\newcommand{\nExpOneReferenceGrid}{401}  
\newcommand{\nExpOneSolverTol}{1\times 10^{-7}}  
\newcommand{\nExpOneSolverVersion}{0.11.1}  
\newcommand{\nExpSixCapViolationTolerance}{1\times 10^{-10}}  
\newcommand{\nExpSixConfidence}{95}  
\newcommand{\nExpSixDelta}{0.05}  
\newcommand{\nExpSixHeadlineEnvelopeBound}{2.395\times 10^{-1}}  
\newcommand{\nExpSixHeadlineHankelBound}{7.648\times 10^{-2}}  
\newcommand{\nExpSixHeadlineProjectionBound}{8.475\times 10^{-3}}  
\newcommand{\nExpSixHeadlineSigma}{1\times 10^{-2}}  
\newcommand{\nExpSixHeadlineWasserstein}{1.689\times 10^{-3}}  
\newcommand{\nExpSixL}{6}  
\newcommand{\nExpSixLastMedianCapExcessPct}{47.80}  
\newcommand{\nExpSixLastNoiseSeed}{5019}  
\newcommand{\nExpSixMaxCapExcessPct}{157.6}  
\newcommand{\nExpSixMaxCondition}{9.24}  
\newcommand{\nExpSixMaxCostGapPct}{10.238}  
\newcommand{\nExpSixMaxSurrogateCapViolation}{0}  
\newcommand{\nExpSixMomentRadius}{1.57986}  
\newcommand{\nExpSixN}{80}  
\newcommand{\nExpSixNoiseSeed}{5000}  
\newcommand{\nExpSixRank}{8}  
\newcommand{\nExpSixReplicates}{20}  
\newcommand{\nExpSixSigmaList}{0, 1\times 10^{-3}, 3\times 10^{-3}, 1\times 10^{-2}, 3\times 10^{-2}, 1\times 10^{-1}}  
\newcommand{\nExpSixSigmaMax}{1\times 10^{-1}}  
\newcommand{\nExpSixSmallestSingular}{2.11958}  
\newcommand{\nExpSixSweepRowsDetailed}{$0$ & $0$ & $0$ & 1.6000189 & 0.00 & 0.00 & 0/20 \\
$1\times 10^{-3}$ & $1.71\times 10^{-4}$ & $8.45\times 10^{-4}$ & 1.6000218 & 0.00 & 1.09 & 8/20 \\
$3\times 10^{-3}$ & $5.10\times 10^{-4}$ & $2.54\times 10^{-3}$ & 1.6000462 & 0.00 & 3.29 & 9/20 \\
$1\times 10^{-2}$ & $1.69\times 10^{-3}$ & $8.47\times 10^{-3}$ & 1.6002889 & 0.00 & 11.33 & 9/20 \\
$3\times 10^{-2}$ & $5.09\times 10^{-3}$ & $2.57\times 10^{-2}$ & 1.6022341 & 2.99 & 36.97 & 12/20 \\
$1\times 10^{-1}$ & $1.77\times 10^{-2}$ & $9.14\times 10^{-2}$ & 1.6275636 & 47.80 & 157.56 & 15/20 \\}  
\newcommand{\nExpSixTrialCount}{120}  
\newcommand{\nExpThreeA}{1 & 0.2 \\ -0.1 & 0.9}  
\newcommand{\nExpThreeB}{1 \\ 0.5}  
\newcommand{\nExpThreeC}{1 & 0}  
\newcommand{\nExpThreeCovRes}{7.80\times 10^{-16}}  
\newcommand{\nExpThreeDegOneRes}{2.62\times 10^{-16}}  
\newcommand{\nExpThreeHankelCols}{75}  
\newcommand{\nExpThreeHankelRows}{12}  
\newcommand{\nExpThreeL}{6}  
\newcommand{\nExpThreeMaxRes}{7.09\times 10^{-15}}  
\newcommand{\nExpThreeMedRes}{2.08\times 10^{-15}}  
\newcommand{\nExpThreeN}{80}  
\newcommand{\nExpThreeNmean}{25}  
\newcommand{\nExpThreeNval}{200}  
\newcommand{\nExpThreePeOrder}{8}  
\newcommand{\nExpThreeRank}{8}  

\newenvironment{keywords}{\par\noindent\textbf{Keywords.}\ }{\par}
\title{A Measure-Theoretic Formulation of Behavioral Systems}
\author{Victor M. Preciado\\
    Department of Electrical and Systems Engineering\\
    Department of Computer and Information Science\\
    University of Pennsylvania\\
    \texttt{preciado@seas.upenn.edu}}
\date{}
\hypersetup{pdftitle={A Measure-Theoretic Formulation of Behavioral Systems},pdfauthor={Victor M. Preciado}}

\begin{document}
	\maketitle

	\begin{abstract}
		Ensemble control specifies probability laws on admissible trajectories,
		including their temporal dependence. Starting from an arbitrary closed
		behavior, we organize its laws into a framework that connects standard probability-measure
		geometry, occupation-measure control, and data-driven trajectory
		parametrization. Classical convexity and extreme-point properties describe
		the feasible laws, while classical Bellman duality governs additive
		control objectives. On compact state and input spaces, full-path constrained
		duality for continuous trajectory functions gives primal attainment and
		zero gap whenever feasible.
		For controllable linear time-invariant systems, the
		noise-free Fundamental Lemma yields an equivalent law-level Hankel
		factorization, with a canonical pseudoinverse lift. The factorization
		transfers trajectory-law constraints exactly to coefficient space and
		connects them to covariance control. Explicit noisy-Hankel bounds
		quantify subspace approximation for second-moment-bounded laws under
		known rank and Gaussian output measurement noise. Numerical
		studies include a nonlinear reset problem whose nonadjacent correlation
		constraint requires temporal memory, covariance-constrained optimization,
		and noisy-data policy evaluation.
		Known transition kernels provide a further structural
		perspective through history-conditional consistency.
	\end{abstract}


	\section{Introduction}

Consider an ensemble of systems with an uncertain initial condition or
observed past. A controller may need to limit output variance and also
preserve a correlation between states at separated times. The marginal
distribution at each time does not specify that correlation. A law on
admissible trajectories records both the ensemble and its temporal
dependence, providing a common object for these control specifications.

Willems' behavioral approach identifies a system with its admissible
trajectories~\cite{Willems1991}. For controllable linear time-invariant (LTI)
systems, the Fundamental Lemma shows that one persistently exciting trajectory
spans the finite-horizon behavior~\cite{Willems2005Fundamental}. This result
underlies data-enabled predictive control (DeePC)~\cite{Coulson2019,Berberich2021}, data-driven
simulation~\cite{MarkovskyRapisarda2008}, and the informativity
framework~\cite{vanWaarde2020Informativity,Trentelman2022Informativity,vanWaarde2023Informativity};
see~\cite{MarkovskyDorfler2021} for a survey.
Exact data parametrizations also exist for special nonlinear classes, for
example systems linear in suitable known input/output
coordinates~\cite{BerberichAllgower2020}. The trajectory set of a general
nonlinear system need not be convex. We instead
study probability laws supported on that set, obtaining a convex description
without averaging incompatible trajectories.

The main difficulty is preserving temporal coupling while using constraints
that are linear in a law. Matching next-state marginals alone can admit
couplings that violate the dynamics. We therefore define the full law through
graph support. Occupation marginals suffice for additive control objectives,
while general path-dependent constraints need temporal information that
these original occupation marginals do not determine. Full trajectory laws
retain that information explicitly. For controllable LTI systems under the
noise-free excitation assumptions, they also admit a distributional
Fundamental-Lemma factorization. For stochastic dynamics, history-conditional
kernel consistency replaces graph support.

The contribution is a unified formulation and its operational consequences.
The probability-measure geometry is standard: convexity, weak closedness,
Dirac extreme points, and density of finite mixtures hold for an arbitrary
closed behavior in a Polish trajectory space
\cite{Billingsley1999,Parthasarathy1967,Phelps2001}. The fixed-initial-law
extreme-point characterization follows from disintegration and fixed-marginal
kernel geometry~\cite{GonzalezHernandezHernandezLerma2005}, with related
mixture representations in strategic-measure theory~\cite{Feinberg1996Strategic}.
We introduce a realization for initial-state slices and
control, and separate probability geometry from parametrization by a Borel
map. Occupation-measure linear programs, Bellman duality,
and constrained Markov decision process duality are classical
\cite{HernandezLermaLasserre1996,Puterman1994,Altman1999}; here they describe
the additive-cost projection of the same trajectory-law set. We also state
full-path constrained duality for continuous trajectory functions.

For controllable LTI systems, the noise-free Fundamental Lemma gives an
equivalent law-level Hankel factorization. The resulting coefficient-law
formulation transfers prescribed marginals and expectation constraints
without changing their values. We use this correspondence for
covariance-constrained ensemble control and give an elementary
projection-based Wasserstein estimate for second-moment-bounded laws. A
rank-truncated noisy Hankel matrix gives an explicit law-set approximation
bound under a specified measurement-noise model. These
deductions organize distributional specifications within the common
behavioral description. The principal operational results are as follows.
\begin{enumerate}
\item Theorem~\ref{thm:path_duality} gives constrained duality for continuous
nonadditive path functionals on an initial-law slice. Its inner problem
minimizes a complete input schedule, with an extreme deterministic
conditional law for every fixed price; primal attainment and zero gap do
not require strict feasibility.
\item The reset example in Subsection~\ref{subsec:path_experiment} exhibits
identical occupation marginals with different cross-time feasibility and
provides an analytic global value. It also distinguishes zero duality gap
from finite multiplier attainment on the unrestricted path space.
\item Proposition~\ref{prop:affine_covariance_control} proves an affine
reduction for a prescribed observed-past law under upper variance caps.
Positive input cost rules out regression-residual covariance at any
optimum, identifying an extreme law in the fixed-past-law slice.
\item Lemma~\ref{lem:robust} and Corollary~\ref{cor:noisy_hankel} connect
subspace perturbation to exact Wasserstein geometry of moment-bounded
law sets and an explicit noisy-Hankel approximation bound under a stated
rank and output-noise model.
\end{enumerate}

Five numerical studies support these connections. Polynomial feasible-set
relaxations enforce the graph identities, the nonlinear control example
tests full-path constraints, and Hankel validation checks law and moment
transfer. Gaussian distributional DeePC tests the affine covariance-control
reduction, and noisy-data approximation tests the projection bounds. The covariance study
compares independently optimized data and model designs, traces the
cost--variance trade-off, and extracts the affine feedback law. The noisy-data
study tests the Wasserstein bounds across repeated measurement errors and
evaluates independently redesigned policies on the true plant.

Occupation-measure methods restore convexity to nonlinear control through
linear flow constraints~\cite{Vinter1993,Lasserre2008,Lasserre2010,LasserreHenrionPrieurTrelat2008,Henrion2014}.
Relaxed controls and Young measures likewise describe per-step
randomization~\cite{Young1969,Warga1972}. These descriptions support additive
cost computation but their marginals do not determine temporal coupling.
Controlled-martingale and occupation-measure linear programming formulations for controlled Markov processes, developed by
Bhatt and Borkar~\cite{BhattBorkar1996} and Kurtz and
Stockbridge~\cite{KurtzStockbridge1998,KurtzStockbridge1999Erratum}, are closer antecedents to our
stochastic formulation. Our focus is a cost-independent finite-horizon
behavioral-measure set and its exact-data deterministic Hankel factorization.

Stochastic behavioral theory already has several formulations. Willems
studies event $\sigma$-algebras in open stochastic
models~\cite{Willems2013Open}; Baggio and Sepulchre develop canonical kernel
representations and interconnection for LTI stochastic
processes~\cite{BaggioSepulchre2017}. Pan, Ou, and
Faulwasser~\cite{PanOuFaulwasser2023} prove a stochastic Fundamental Lemma
that includes a random-coefficient Hankel representation of square-integrable
trajectories. Our law-level factorization is its deterministic-dynamics
counterpart at the level of distributions; the comparison is made precisely
after Theorem~\ref{thm:measure_level_FL}. The broader stochastic behavioral
setting is surveyed in~\cite{Faulwasser2023Behavioral}.

Distributionally robust DeePC protects a coefficient-vector decision against
ambiguity in noisy-data laws~\cite{CoulsonLygerosDorfler2022}. Our measure-level
theorem establishes an exact noise-free feasible-set correspondence between
coefficient laws and trajectory laws. Covariance steering prescribes
moments for linear stochastic dynamics~\cite{ChenGeorgiouPavon2016}.
Data-driven covariance steering already uses the Fundamental Lemma to
parametrize mean and covariance control~\cite{PilipovskyTsiotras2023}, with
robust formulations for uncertain data~\cite{PilipovskyTsiotras2024,PilipovskyTsiotras2026DUST} and a
recent output-feedback extension~\cite{MoustroufisTsiotras2026}. Our Gaussian
study illustrates the coefficient-law correspondence for an observed past
ensemble and upper output-variance constraints under deterministic dynamics.
It is an application of this correspondence within established covariance
control, rather than a new covariance-steering methodology.

Section~\ref{sec:behavioral_measures} develops the structural theory,
Section~\ref{sec:optimal_control} treats control and duality, and
Section~\ref{sec:structural} covers compactness, the LTI specialization, and
the stochastic extension. Section~\ref{sec:numerics} presents the experiments.

\section{Behavioral Measures}
\label{sec:behavioral_measures}

We first define laws on a closed behavior without choosing a realization.
We then specialize to state-space paths to describe graph support, occupation
marginals, and a prescribed initial-state law. For any Polish space $S$,
$\calP(S)$ denotes its Borel probability measures. For a Borel map $\varphi$,
the pushforward is $\varphi_{\#}\mu(A)=\mu(\varphi^{-1}(A))$.

\subsection{Laws on Closed Behaviors}
\label{subsec:closed_behaviors}

\begin{definition}[Behavioral-measure set]
\label{def:behavioral_measure}
Let $S$ be a Polish trajectory space. A closed finite-horizon behavior is a
nonempty closed subset $B\subseteq S$ of admissible trajectories. Closedness
means that if a sequence of trajectories in $B$ converges in $S$, its limit
also belongs to $B$; it does not require $B$ to be compact. Its
behavioral-measure set is
\[
\calM(B):=\calP(B)
=\{\mu\in\calP(S):\mu(B)=1\}.
\]
For external signals in a Polish space $\calW$, one may take
$S=\calW^L$ for a finite window of length $L$. The behavior $B$ may be nonlinear and
nonconvex.
\end{definition}

This definition depends only on admissible trajectories and is independent
of a control objective or a state-space representation. The following
standard probability-measure facts provide its structural theory.

\begin{lemma}[Geometry of laws on a closed behavior]
\label{lem:closed_behavior}
Let $B$ be a nonempty closed subset of a Polish space $S$. Then $\calM(B)$
is convex and weakly closed in $\calP(S)$, and is weakly compact if $B$ is
compact. Its extreme points are exactly $\delta_w$ for $w\in B$, and its
finitely supported measures are weakly dense.
\end{lemma}

\begin{proof}
Convexity follows from the affine condition $\mu(B)=1$. If
$\mu_k\rightharpoonup\mu$ with $\mu_k(B)=1$, Portmanteau
\cite[Thm.~2.1]{Billingsley1999} gives
$1=\limsup_k\mu_k(B)\le\mu(B)$, hence $\mu(B)=1$.
If $B$ is compact, probabilities supported on $B$ form a tight family;
closedness and Prokhorov's theorem give weak compactness
\cite{Billingsley1999,Parthasarathy1967}.

A convex decomposition of $\delta_w$ forces both component laws to assign
mass one to $\{w\}$, so it is trivial. A non-Dirac law $\mu\in\calP(B)$
has a Borel set $A\subseteq B$ with $0<\mu(A)<1$, and splits nontrivially
into its normalized restrictions to $A$ and $B\setminus A$. This identifies
the extreme points~\cite{Phelps2001}. Finally, $B$ is Polish, so finite
atomic approximations are weakly dense~\cite{Parthasarathy1967}.
\end{proof}

A parametrization is a separate tool for representing these laws. The next
lemma requires measurability and a chosen right inverse; its conclusion
imposes no continuity or closedness requirement on the image.

\begin{lemma}[Borel lifting through a parametrization]
\label{lem:param}
Let $P,S$ be Polish spaces and $\Phi:P\to S$ be Borel. Suppose
$E:=\Phi(P)$ is Borel and $\sigma:E\to P$ is a Borel right inverse of
$\Phi$. Identifying $\calP(E)$ with laws in $\calP(S)$ supported on $E$,
we have
\[
\calP(E)=\{\Phi_{\#}\nu:\nu\in\calP(P)\}.
\]
For every $\mu\in\calP(E)$, $\Phi_{\#}(\sigma_{\#}\mu)=\mu$.
Thus $\sigma_{\#}\mu$ is a lift associated with the chosen right inverse.
\end{lemma}

\begin{proof}
Every $\Phi_{\#}\nu$ assigns mass one to $E$. Conversely, $\sigma$ is
Borel and $\Phi\circ\sigma=\id_E$, so
$\Phi_{\#}(\sigma_{\#}\mu)=(\id_E)_{\#}\mu=\mu$.
\end{proof}

If $\Phi$ is injective, its image is Borel and its inverse on that image is
Borel by the Lusin--Souslin theorem~\cite[Cor.~15.2]{Kechris1995}. In that
case the right-inverse hypotheses of Lemma~\ref{lem:param} hold
automatically; continuity is still unnecessary.

\subsection{State-Space Realizations and Graph Support}
\label{subsec:realized_behaviors}

We now introduce a realization to define state coordinates and initial-law
slices. The following assumptions govern the realized constructions below
and the control problems of Section~\ref{sec:optimal_control}.

	\begin{assumption}[State-space conventions]
		\label{ass:standing}
		Fix a finite horizon $T\in\mathbb{N}$. The state space $\calX$, input space
		$\calU$, and output space $\calY$ are nonempty Polish spaces. The dynamics map $f:\calX\times\calU\to\calX$ and output map
		$h:\calX\times\calU\to\calY$ are continuous.
		Compactness assumptions are introduced explicitly when needed for the
		Bellman-duality and existence results in
		Sections~\ref{sec:optimal_control}--\ref{sec:structural}.
	\end{assumption}

	Consider the discrete-time controlled system
	\begin{equation}
		\label{eq:dynamics}
		x_{t+1} = f(x_t,u_t), \qquad y_t = h(x_t,u_t), \quad t=0,\dots,T-1,
	\end{equation}
	where $x_t\in\calX$ is the state, $u_t\in\calU$ is the control input, and
	$y_t\in\calY$ is the output. We work on the finite-horizon trajectory space
	\[
	\Omega_T := \calX^{T+1}\times \calU^T\times \calY^T,
	\]
	equipped with its product Borel $\sigma$-algebra. For a trajectory
	$\omega\in\Omega_T$, write $x_t,u_t,y_t$ for its coordinates. The canonical
	projections are $X_t(\omega)=x_t$, $U_t(\omega)=u_t$, and $Y_t(\omega)=y_t$.
	States are indexed through $T$; inputs and outputs are indexed through
	$T-1$, with no output defined at time $T$.

\begin{definition}[Realized path set and initial-law slice]
\label{def:realized_behavior}
The admissible state-space path set is
\[
\ifdim\linewidth>350pt
\def\realizedpathdisplaybreak{\quad}
\else
\def\realizedpathdisplaybreak{\\&\quad}
\fi
\begin{aligned}
&\mathfrak{B}_T:=\{\omega\in\Omega_T:
X_{t+1}(\omega)=f(X_t(\omega),U_t(\omega)),\realizedpathdisplaybreak
Y_t(\omega)=h(X_t(\omega),U_t(\omega)),\quad t=0,\ldots,T-1\}.
\end{aligned}
\]
We abbreviate its law set by
$\calM_\calB:=\calM(\mathfrak{B}_T)=\calP(\mathfrak{B}_T)$.
For $\rho_0\in\calP(\calX)$, its initial-law slice is
\[
\calM_\calB(\rho_0)
:=\{\mu\in\calM_\calB:(X_0)_{\#}\mu=\rho_0\}.
\]
A fixed initial state $\bar x$ corresponds to $\rho_0=\delta_{\bar x}$.
\end{definition}

Each graph equality is the inverse image of a closed diagonal under a
continuous map, so $\mathfrak{B}_T$ is closed in $\Omega_T$. It is nonempty
because an initial state and an input sequence generate a trajectory.
Lemma~\ref{lem:closed_behavior} therefore gives the geometry of
$\calM_\calB$ directly.

\begin{corollary}[Initial-law slices]
\label{cor:initial_slice}
Under Assumption~\ref{ass:standing}, for every $\rho_0\in\calP(\calX)$,
$\calM_\calB(\rho_0)$ is nonempty,
convex, and weakly closed in $\calP(\Omega_T)$. It is weakly compact when
$\mathfrak{B}_T$ is compact.
\end{corollary}

\begin{proof}
A fixed input sequence gives a continuous trajectory map of the initial
state, which pushes $\rho_0$ to a law in the slice. The map
$\mu\mapsto(X_0)_{\#}\mu$ is affine and weakly continuous.
Intersecting its inverse image of $\{\rho_0\}$ with $\calM_\calB$ proves
the remaining claims by Lemma~\ref{lem:closed_behavior}.
\end{proof}

\begin{remark}[Realizations and external projections]
\label{rem:realization_dependence}
Definition~\ref{def:behavioral_measure} is realization-free. Choosing
$\mathfrak{B}_T$ retains internal state coordinates and gives meaning to
initial-state slices and state-based control. Its external image under a
continuous projection need not be closed on noncompact spaces. For example,
$\calX=\calY=\R$, $\calU=\{0\}$, $f(x,0)=x$, and $h(x,0)=e^x$ give a closed
realized graph whose one-step external outputs form $(0,\infty)$ in $\R$.
Thus closedness of an external behavior must be assumed or established
separately. Compact signal spaces make the external image compact and
closed; the finite-dimensional LTI case is treated in
Subsection~\ref{subsec:lti_fundamental_lemma}.
\end{remark}

The full law set contains all initial distributions. Fixing $\rho_0$ selects
a slice without changing the behavior. Although the full set has Dirac
extreme points, an extreme point of a slice need not be a Dirac law.
Proposition~\ref{prop:slice} characterizes those extreme points.

	Linearity in a measure means that $f$ and $h$ stay fixed while the law
assigning probabilities to trajectory events varies. For $\theta\in[0,1]$ and
	$\mu^\theta=\theta\mu^a+(1-\theta)\mu^b$, any fixed bounded Borel function
	$F$ satisfies
	\[
	\int F\,d\mu^\theta
	=\theta\int F\,d\mu^a+(1-\theta)\int F\,d\mu^b.
	\]
	This applies to the bounded mismatch functions below even when $f$ is nonlinear. If both laws
	assign probability one to admissible trajectories, so does their mixture.
	For example, the scalar graph $x^+=x^2+u$ contains the triples $(0,0,0)$ and
	$(1,0,1)$. An equally weighted law on these triples is graph-supported,
	although its mean triple $(1/2,0,1/2)$ fails the graph equation.

	For an arbitrary trajectory law $\mu$, introduce the joint transition
	marginal
	\[
	\eta_t:=(X_t,U_t,X_{t+1})_{\#}\mu
	\in\calP(\calX\times\calU\times\calX).
	\]
	This measure contains the present state, input, and next state jointly;
	a marginal on $(X_t,U_t,Y_t)$ alone does not contain the next-state coordinate.
	The operators below abbreviate differences between test-function values at
	the actual and predicted next states or outputs.

	Although the behavioral-measure set is defined by a support condition on the
	admissible path set, it also satisfies a family of weak operator identities.
		These identities play a central role in the moment--sum-of-squares (SOS) relaxations of
		Section~\ref{sec:numerics} and the duality results of
		Section~\ref{sec:optimal_control}.

		Let $C_b(\calX)$ and $C_b(\calY)$ denote the spaces of bounded continuous
		real-valued functions on $\calX$ and $\calY$, respectively, and let
		$d_{\calX}$ and $d_{\calY}$ denote compatible bounded metrics on
		$\calX$ and $\calY$.

	\begin{proposition}[Graph support and operator identities]
		\label{prop:operator_identities}
		For each $t=0,\dots,T-1$, let $\varphi\in C_b(\calX)$ and
		$\psi\in C_b(\calY)$ be arbitrary bounded continuous test functions on the
		state and output spaces, respectively. Define the operators
		$\mathcal{L}_t:C_b(\calX)\to C_b(\Omega_T)$ and
		$\mathcal{O}_t:C_b(\calY)\to C_b(\Omega_T)$ by
		\begin{align}
			\label{eq:L_operator}
			(\mathcal{L}_t\varphi)(\omega)
			&:= \varphi(X_{t+1}(\omega))
			- \varphi\bigl(f(X_t(\omega),U_t(\omega))\bigr),\\
			\label{eq:H_operator}
			(\mathcal{O}_t\psi)(\omega)
			&:= \psi(Y_t(\omega))
			- \psi\bigl(h(X_t(\omega),U_t(\omega))\bigr).
		\end{align}
		For every Borel probability measure on the trajectory space,
		$\mu\in\calP(\Omega_T)$, the following conditions satisfy
		$\textup{(i)}\Leftrightarrow\textup{(ii)}\Rightarrow\textup{(iii)}$.
		The reverse implication
		$\textup{(iii)}\Rightarrow\textup{(i)}$ fails in general.
		\textup{(i)} \emph{Graph support:} $\mu$ is supported on the admissible path set, i.e.,
			$\mu\in\mathcal{M}_{\mathcal{B}}$ or equivalently $\mu(\mathfrak{B}_T)=1$.

\textup{(ii)} \emph{Metric residual:} the average squared mismatch between the actual
			and predicted states and outputs is zero under $\mu$. That is, for every
			$t=0,\dots,T-1$,
			\ifdim\linewidth>350pt
\begin{equation}
\label{eq:nonnegative_residual}
\int_{\Omega_T}\Bigl[d_{\calX}\bigl(X_{t+1},f(X_t,U_t)\bigr)^2
+d_{\calY}\bigl(Y_t,h(X_t,U_t)\bigr)^2\Bigr]\,d\mu=0.
\end{equation}
\else
\begin{equation}
				\label{eq:nonnegative_residual}
\begin{aligned}
\int_{\Omega_T}\Bigl[&d_{\calX}\bigl(X_{t+1},f(X_t,U_t)\bigr)^2\\
&{}+d_{\calY}\bigl(Y_t,h(X_t,U_t)\bigr)^2\Bigr]\,d\mu=0.
\end{aligned}
			\end{equation}
\fi

\textup{(iii)} \emph{Weak operator identities:} the distributions of $X_{t+1}$ and
			$f(X_t,U_t)$ under $\mu$ agree, as do those of $Y_t$ and $h(X_t,U_t)$. That
			is, for every $t=0,\dots,T-1$, every $\varphi\in C_b(\calX)$, and every
			$\psi\in C_b(\calY)$,
			\begin{equation}
				\label{eq:weak_consistency}
				\int_{\Omega_T}\mathcal{L}_t\varphi\,d\mu = 0,
				\qquad
				\int_{\Omega_T}\mathcal{O}_t\psi\,d\mu = 0.
			\end{equation}

	\end{proposition}

	\begin{proof}
		For each $t=0,\dots,T-1$, define
\ifdim\linewidth>350pt
\[
g_t(\omega):=d_{\calX}\bigl(X_{t+1}(\omega),f(X_t(\omega),U_t(\omega))\bigr)^2
+d_{\calY}\bigl(Y_t(\omega),h(X_t(\omega),U_t(\omega))\bigr)^2.
\]
\else
\[
\begin{aligned}
g_t(\omega):={}&
d_{\calX}\bigl(X_{t+1}(\omega),f(X_t(\omega),U_t(\omega))\bigr)^2\\
&{}+d_{\calY}\bigl(Y_t(\omega),h(X_t(\omega),U_t(\omega))\bigr)^2.
\end{aligned}
\]
\fi
		This function is continuous and nonnegative, and $g_t(\omega)=0$ if and
		only if
		$X_{t+1}(\omega)=f(X_t(\omega),U_t(\omega))$ and
		$Y_t(\omega)=h(X_t(\omega),U_t(\omega))$. Hence
		${\mathfrak{B}_T = \bigcap_{t=0}^{T-1} g_t^{-1}(\{0\})}$.

		\textup{(i)}$\Leftrightarrow$\textup{(ii).}
		If $\mu(\mathfrak{B}_T)=1$, then $g_t=0$ $\mu$-almost surely for every
		$t$, so~\eqref{eq:nonnegative_residual} holds. Conversely,
		\eqref{eq:nonnegative_residual} and $g_t\ge 0$ imply
		$g_t=0$ $\mu$-almost surely for every $t$.
		Since there are finitely many time indices,
		$\mu(\mathfrak{B}_T) = \mu\bigl(\bigcap_{t} g_t^{-1}(\{0\})\bigr) = 1$.

		\textup{(ii)}$\Rightarrow$\textup{(iii).}
		If~\eqref{eq:nonnegative_residual} holds, then
		$X_{t+1}=f(X_t,U_t)$ and $Y_t=h(X_t,U_t)$ $\mu$-almost surely, so
		$\mathcal{L}_t\varphi=\mathcal{O}_t\psi=0$ $\mu$-almost surely for all
		$\varphi\in C_b(\calX)$ and $\psi\in C_b(\calY)$. Integrating yields
		\eqref{eq:weak_consistency}.

		\textup{(iii)}$\not\Rightarrow$\textup{(i).}
		Take $T=1$, $\calX=\calU=[0,1]$, $\calY=\{0\}$, $f(x,u)=u$, and
		$h(x,u)=0$. Let $X_0=0$ deterministically, let
		$U_0\sim\mathrm{Unif}[0,1]$, and let $X_1\sim\mathrm{Unif}[0,1]$ be
		independent of $U_0$. Denote by $\mu$ the induced law on $\Omega_1$.
		Since $X_1$ and $U_0=f(X_0,U_0)$ have the same distribution,
		$\int \mathcal{L}_0\varphi\,d\mu
		= \mathbb{E}[\varphi(X_1)] - \mathbb{E}[\varphi(U_0)] = 0$
		for every $\varphi\in C_b([0,1])$. The $\mathcal{O}_0$ identity holds
		trivially. Thus~\eqref{eq:weak_consistency} holds, but
		$\mu(\mathfrak{B}_1)=\mathbb{P}(X_1=U_0)=0$ since $X_1$ and $U_0$ are
		independent continuous random variables. The weak identities match
		marginals but do not enforce graph support.
	\end{proof}

	The dynamics identity can equivalently be evaluated on the transition marginal:
	\[
	\int_{\Omega_T}\mathcal L_t\varphi\,d\mu
	=\int_{\calX\times\calU\times\calX}
	[\varphi(x^+)-\varphi(f(x,u))]\,d\eta_t.
	\]
	Thus the test-function operator acts on trajectories, and its integral depends
	on the joint transition law. Equality of the next-state marginal
$(X_{t+1})_\#\mu$ and the predicted-state marginal $f_\#\lambda_t$ does
not identify the joint law of $(X_t,U_t,X_{t+1})$.

	\begin{remark}[Multiplied identities and graph support]
		\label{rem:operator_outer_approx}
		The weak identities~\eqref{eq:weak_consistency} only match marginals.
		Retaining joint tests recovers graph support: for each $t$ and every
		$q\in C_b(\calX\times\calU)$ and $\varphi\in C_b(\calX)$ require
		\[
		\int q(X_t,U_t)\bigl[\varphi(X_{t+1})-
		\varphi(f(X_t,U_t))\bigr]\,d\mu=0,
		\]
		with the analogous identity for $Y_t$ and $h$. Products of bounded
		continuous test functions determine the joint transition law, so these
		identities identify the conditional next-state law as
		$\delta_{f(X_t,U_t)}$ and the output law as $\delta_{h(X_t,U_t)}$.
		Intersecting over the finitely many time indices gives graph support.
		Section~\ref{sec:extensions} replaces evaluation at $f(X_t,U_t)$ by a
		kernel average and tests against continuous functions of the available
		history and current control.

		For polynomial dynamics on compact support, moment-SOS relaxations impose
		the graph-ideal equalities
		\[
		\int q(x_t,u_t,x_{t+1})
		\bigl(x_{t+1}-f(x_t,u_t)\bigr)\,d\mu=0
		\]
		and their output counterparts for monomial multipliers up to the
		relaxation degree, together with moment and localizing positive semidefinite (PSD) constraints
		\cite{Lasserre2008,Lasserre2010}. A finite truncation remains an outer
		approximation; the unmultiplied marginal identities do not suffice.
	\end{remark}

	For optimal control it is useful to extract per-step information from a
	behavioral measure. The next proposition provides the bridge to the
	occupation-measure formulation of Section~\ref{sec:optimal_control}.

	\begin{proposition}[Occupation representation and reconstruction]
		\label{prop:occupation_representation}
		Fix an initial distribution $\rho_0\in\calP(\calX)$ and a behavioral measure
		$\mu\in\mathcal{M}_{\mathcal{B}}(\rho_0)$. Let $\pi_X:\calX\times\calU\to
		\calX$ denote the projection $\pi_X(x,u)=x$. For each $t=0,\dots,T-1$,
		define the state distribution at time $t$ and the joint state-input
		distribution at time $t$ as
		\[
		\rho_t := (X_t)_{\#}\mu \in \calP(\calX), \qquad
		\lambda_t := (X_t,U_t)_{\#}\mu \in \calP(\calX\times\calU).
		\]
		These distributions satisfy the following three flow constraints for every
		$t=0,\dots,T-1$:
		\textup{(i)} The state distribution $\rho_t$ is the marginal of
			$\lambda_t$ over the input:
			\begin{equation}
				\label{eq:flow_marginal}
				(\pi_X)_{\#}\lambda_t = \rho_t.
			\end{equation}

\textup{(ii)} The state distribution propagates through the dynamics:
			\begin{equation}
				\label{eq:flow_dynamics}
				\rho_{t+1} = f_{\#}\lambda_t.
			\end{equation}

\textup{(iii)} The output is determined by the state and input through $h$:
			\begin{equation}
				\label{eq:flow_output}
				(X_t,U_t,Y_t)_{\#}\mu = (\mathrm{id},h)_{\#}\lambda_t.
			\end{equation}

			Conversely, suppose measures $\{\rho_t\}_{t=0}^{T}$ and
			$\{\lambda_t\}_{t=0}^{T-1}$ satisfy~\eqref{eq:flow_marginal}
			and~\eqref{eq:flow_dynamics} for a prescribed $\rho_0$. Then, there exists a behavioral measure
			$\mu\in\mathcal{M}_{\mathcal{B}}(\rho_0)$ whose state and state-input
			marginals are exactly $\rho_t$ and $\lambda_t$, and whose output marginals
			satisfy~\eqref{eq:flow_output}.
	\end{proposition}

	\begin{proof}
		Let $\mu\in\mathcal{M}_{\mathcal{B}}(\rho_0)$. Identity
		\eqref{eq:flow_marginal} is immediate from the definitions. For any
		bounded continuous $\varphi\in C_b(\calX)$,
\ifdim\linewidth>350pt
\[
\int_{\calX}\varphi(x)\,d\rho_{t+1}(x)
=\int_{\Omega_T}\varphi(X_{t+1})\,d\mu
=\int_{\Omega_T}\varphi(f(X_t,U_t))\,d\mu
=\int_{\calX\times\calU}\varphi(f(x,u))\,d\lambda_t(x,u),
\]
\else
\[
\begin{aligned}
			\int_{\calX}\varphi(x)\,d\rho_{t+1}(x)
			&= \int_{\Omega_T}\varphi(X_{t+1})\,d\mu \\
			&= \int_{\Omega_T}\varphi(f(X_t,U_t))\,d\mu \\
			&= \int_{\calX\times\calU}\varphi(f(x,u))\,d\lambda_t(x,u),
\end{aligned}
\]
\fi
		where the second equality uses $\mu(\mathfrak{B}_T)=1$. Hence
		\eqref{eq:flow_dynamics} holds. Likewise, because
		$Y_t=h(X_t,U_t)$ $\mu$-almost surely,
		\[
		\int_{\Omega_T}\zeta(X_t,U_t,Y_t)\,d\mu
		= \int_{\calX\times\calU}\zeta(x,u,h(x,u))\,d\lambda_t(x,u)
		\]
		for every bounded Borel $\zeta$, which proves~\eqref{eq:flow_output}.

		Conversely, disintegration~\cite[Ch.~8]{Kallenberg2021}
		yields Borel stochastic kernels $\kappa_t(\cdot\mid x)$ such that
		\[
		\lambda_t(dx,du)=\rho_t(dx)\,\kappa_t(du\mid x), \qquad t=0,\dots,T-1.
		\]
			The standard finite-horizon Ionescu--Tulcea construction then yields
		\begin{equation}
			\label{eq:markov_policy_lift}
\ifdim\linewidth>350pt
\mu(dx_{0:T},du_{0:T-1},dy_{0:T-1})
:=\rho_0(dx_0)\prod_{t=0}^{T-1}
\Bigl[\kappa_t(du_t\mid x_t)\,
\delta_{f(x_t,u_t)}(dx_{t+1})\,
\delta_{h(x_t,u_t)}(dy_t)\Bigr].
\else
\begin{aligned}
&\mu(dx_{0:T},du_{0:T-1},dy_{0:T-1})\\
&\quad:=\rho_0(dx_0)\prod_{t=0}^{T-1}
\Bigl[\kappa_t(du_t\mid x_t)\\
&\hphantom{\quad:=\rho_0(dx_0)\prod_{t=0}^{T-1}\Bigl[}
\delta_{f(x_t,u_t)}(dx_{t+1})\,
\delta_{h(x_t,u_t)}(dy_t)\Bigr].
\end{aligned}
\fi
		\end{equation}
		This measure is supported on $\mathfrak{B}_T$. Moreover, by induction using
		\eqref{eq:flow_marginal} and \eqref{eq:flow_dynamics}, its state and
		state-input marginals are $\rho_t$ and $\lambda_t$, respectively. Its
		$(X_t,U_t,Y_t)$-marginals are $(\mathrm{id},h)_{\#}\lambda_t$.
	\end{proof}

	\begin{remark}[Non-uniqueness of reconstruction]
		\label{rem:nonuniqueness_reconstruction}
		The reconstruction \eqref{eq:markov_policy_lift} is the canonical
		Markov-policy lift induced by the kernels $\kappa_t(du\mid x)$. It
		reproduces the prescribed one-step marginals but need not preserve the
		temporal couplings of an arbitrary behavioral measure with the same
		per-step marginals.
	\end{remark}

\begin{proposition}[Range of a path functional at fixed occupations]
\label{prop:occupation_fiber}
Under Assumption~\ref{ass:standing}, for a feasible occupation tuple
$(\rho_t,\lambda_t)$ of Proposition~\ref{prop:occupation_representation},
let
\[
\ifdim\linewidth>350pt
\mathcal F_{\rho,\lambda}:=\{\mu\in\calM_\calB(\rho_0):\;
(X_t,U_t)_\#\mu=\lambda_t\ (t<T),\quad (X_T)_\#\mu=\rho_T\}.
\else
\begin{aligned}
\mathcal F_{\rho,\lambda}:=\{\mu\in\calM_\calB(\rho_0):\;&
(X_t,U_t)_\#\mu=\lambda_t\ (t<T),\\
& (X_T)_\#\mu=\rho_T\}.
\end{aligned}
\fi
\]
For any $c\in C_b(\mathfrak B_T)$, the attainable expectations form a compact
interval $[c_-,c_+]$, where
\[
c_-:=\min_{\mu\in\mathcal F_{\rho,\lambda}}\int c\,d\mu,
\qquad
c_+:=\max_{\mu\in\mathcal F_{\rho,\lambda}}\int c\,d\mu.
\]
Thus two full-law linear programs quantify the information about $c$ that
the fixed occupation marginals leave undetermined. If $\mathfrak B_T$ is
compact, every continuous real-valued $c$ is allowed.
\end{proposition}

\begin{proof}
Reconstruction makes $\mathcal F_{\rho,\lambda}$ nonempty. The fixed
occupation tuple fixes every state and input coordinate marginal, as well
as $(Y_t)_\#\mu=h_\#\lambda_t$. Tightness of these finitely many prescribed
marginals gives uniform tightness on the closed Polish path set
$\mathfrak B_T$. The marginal constraints are weakly closed and affine,
so Prokhorov's theorem makes the fiber compact and convex. Its image under the
continuous affine map $\mu\mapsto\int c\,d\mu$ is a nonempty compact
convex subset of $\R$, which is the claimed interval with attained
endpoints. For an additive stage/terminal functional with finite component
expectations, the interval is a singleton because its expectation is fixed
by the occupation tuple.
\end{proof}

		For a behavioral law, the joint transition marginal is fixed by
		$\lambda_t$ and the graph of $f$:
		\[
		\eta_t=(\mathrm{id}_{\calX\times\calU},f)_{\#}\lambda_t.
		\]
		Its state-input marginal is $\lambda_t$ and its next-state marginal is
		$\rho_{t+1}$. This identifies the joint object behind the flow equations.

	The behavioral-measure set admits two complementary descriptions: the
	graph-support description of Definition~\ref{def:realized_behavior}, and the
	parametrization by initial condition and inputs through the map $\Phi_T$
	introduced in Subsection~\ref{subsec:initial_slice} (Corollary~\ref{cor:homeo}). The graph-support description remains the primary
	object. The parametrization $\Phi_T$ composes the dynamics $T$ times, so its
	polynomial degree can increase with the horizon, whereas the graph constraints
	preserve the degree of $f$ and $h$ and enter the Lasserre hierarchy at fixed
		degree per time step. Subsection~\ref{subsec:path_experiment} compares the
full and reduced graph moment blocks, of sizes
$\nExpFiveFullBlock\times\nExpFiveFullBlock$ and
$\nExpFiveGraphBlock\times\nExpFiveGraphBlock$, for one such instance.
The full trajectory-law viewpoint extends to the
		transition kernels of Section~\ref{sec:extensions} by replacing graph support
		with history-conditional kernel consistency. It also keeps the formulation
		close to the behavioral viewpoint by avoiding a distinguished choice of
		inputs as free variables.

\subsection{Initial-Law Parametrization and Extreme Points}
\label{subsec:initial_slice}

Under Assumption~\ref{ass:standing}, let
$\Phi_T:\calX\times\calU^T\to\Omega_T$ send an initial state and input
schedule to the trajectory generated by~\eqref{eq:dynamics}. Its image is
exactly $\mathfrak{B}_T$. It is continuous and injective, with continuous
inverse on that image given by the coordinate projection
$\sigma_T(\omega)=(X_0(\omega),U_{0:T-1}(\omega))$.
Lemma~\ref{lem:param} applies to this parametrization. Continuity of both
maps yields the stronger topological conclusion below.

\begin{corollary}[Affine homeomorphism of the parametrization]
\label{cor:homeo}
The pushforward
$(\Phi_T)_{\#}:\calP(\calX\times\calU^T)\to\calM_\calB$ is an affine
bijection and a homeomorphism for weak convergence, with inverse
$(\sigma_T)_{\#}$. If $p_\calX(x,u)=x$ and
\[
\Pi(\rho_0):=\{\pi\in\calP(\calX\times\calU^T):
(p_\calX)_{\#}\pi=\rho_0\},
\]
then it maps $\Pi(\rho_0)$ bijectively onto $\calM_\calB(\rho_0)$.
\end{corollary}

\begin{proof}
The maps $\Phi_T$ and $\sigma_T$ are mutually inverse between
$\calX\times\calU^T$ and $\mathfrak{B}_T$, so their pushforwards are
mutually inverse affine maps. Both are weakly continuous because the
underlying maps are continuous. The slice correspondence follows from
$X_0\circ\Phi_T=p_\calX$.
\end{proof}

The next fixed-marginal characterization is consistent with classical
deterministic-kernel extremality and strategic-measure mixture
representations~\cite{GonzalezHernandezHernandezLerma2005,Feinberg1996Strategic}.
We give the measurable perturbation proof for this particular slice.

	\begin{proposition}[Extreme points of the initial slice]
		\label{prop:slice}
		Let Assumption~\ref{ass:standing} hold and let $\rho_0\in\calP(\calX)$. A
		measure $\mu\in\mathcal{M}_{\mathcal{B}}(\rho_0)$ is an extreme point of
		$\mathcal{M}_{\mathcal{B}}(\rho_0)$ if and only if there is a Borel
		measurable map $g:\calX\to\calU^T$ with $U_{0:T-1}=g(X_0)$ $\mu$-almost
		surely; equivalently, $\mu=(\Phi_T(\cdot,g(\cdot)))_{\#}\rho_0$. The
		extreme points of the slice are exactly the behavioral measures generated
		by deterministic open-loop input laws of the initial condition, understood
		up to $\rho_0$-null modifications. When
		$\rho_0=\delta_{\bar x}$ they reduce to the Dirac trajectory measures of
		Lemma~\ref{lem:closed_behavior}.
	\end{proposition}

\begin{proof}
The affine bijection of Corollary~\ref{cor:homeo} preserves extreme points.
It therefore suffices to characterize the extreme points of $\Pi(\rho_0)$.
Disintegrate $\pi(dx,du)=\rho_0(dx)\kappa(du\mid x)$
\cite[Ch.~8]{Kallenberg2021}. If $\kappa(\cdot\mid x)=\delta_{g(x)}$
almost everywhere, uniqueness of disintegration and extremality of Dirac
masses show that every convex decomposition of $\pi$ in $\Pi(\rho_0)$ is
trivial.

Conversely, fix a countable base $\{A_j\}$ of $\calU^T$. A non-Dirac
probability on a Polish space has two distinct support points, so some
basic open set has probability strictly between zero and one. Hence the
non-Dirac set
\[
D=\bigcup_j\{x:0<\kappa(A_j\mid x)<1\}
\]
is Borel. If $\rho_0(D)>0$, choose $A=A_j$ so that
$E:=\{x:0<\kappa(A\mid x)<1\}$ has positive measure. Set
$p(x)=\kappa(A\mid x)$ and $q(x,u)=\1_E(x)(\1_A(u)-p(x))$, and define
$\pi^\pm=(1\pm q)\pi$. Since $|q|\le1$ and
$\int q(x,u)\kappa(du\mid x)=0$, both $\pi^\pm$ belong to $\Pi(\rho_0)$
and $\pi=(\pi^++\pi^-)/2$. They are distinct because
\[
\pi^+(E\times A)-\pi^-(E\times A)
=2\int_E p(x)(1-p(x))\,d\rho_0(x)>0.
\]
Therefore an extreme point has a Dirac conditional kernel almost everywhere.

When $\rho_0(D)=0$, let $g(x)$ be the atom on $D^c$ and assign a fixed
$u_\ast\in\calU^T$ on $D$. For every open $O\subseteq\calU^T$,
\[
\{x\in D^c:g(x)\in O\}
=D^c\cap\{x:\kappa(O\mid x)=1\}.
\]
Thus $g$ is Borel everywhere, and the trajectory-map correspondence gives
the stated representation.
\end{proof}

	\begin{remark}[Deterministic input laws]
		\label{rem:slice}
		Proposition~\ref{prop:slice} identifies deterministic input schedules
		conditional on the initial state. It does not assert that every such
		schedule can be implemented by a memoryless Markov feedback after
		trajectories from different initial states merge. The existence of a
		Markov optimal selector for the additive control problem is established
		separately in Corollary~\ref{cor:policy_extraction}.
	\end{remark}

\section{Optimal Control Over Behavioral Measures}
	\label{sec:optimal_control}

We first study additive costs using occupation measures and Bellman recursion.
We then optimize continuous functions of the entire trajectory, subject to
expectation constraints. Occupation marginals preserve additive quantities;
the full law also retains temporal couplings between nonadjacent stages.

	We now make this precise. Let $\rho_0\in\calP(\calX)$ be a prescribed initial
	distribution, and let $\ell:\calX\times\calU\to\R$ and $\phi:\calX\to\R$
	denote Borel stage and terminal costs, both bounded below so that their expected
	values are defined. The behavioral formulation of
	finite-horizon optimal control seeks the trajectory distribution in the
	initial slice $\mathcal{M}_{\mathcal{B}}(\rho_0)$ that minimizes the expected
	total cost:
	\begin{equation}
		\label{eq:path_primal}
		\inf_{\mu\in\mathcal{M}_{\mathcal{B}}(\rho_0)}
		\int_{\Omega_T}\Bigl(\sum_{t=0}^{T-1}\ell(X_t,U_t)+\phi(X_T)\Bigr)\,d\mu.
	\end{equation}
		The optimization is over probability measures on entire trajectories. However,
		the two-way correspondence established in
		Proposition~\ref{prop:occupation_representation} allows us to reformulate this
		path-level problem equivalently as an optimization over the per-step state and
		state-input distributions $\rho_t$ and $\lambda_t$. This reduction is valid
		here because the cost is additive and depends only on $(X_t,U_t)$ and $X_T$;
		general path-dependent costs or temporal-correlation constraints are not
determined by these original occupation marginals. State augmentation can
retain the needed memory; storing one earlier state need not grow in
dimension with the delay, and retaining the whole history handles arbitrary
finite-horizon path functions. The full-law formulation instead keeps one
common trajectory space on which new path specifications are linear
functionals without redesigning state coordinates. We record this
		reformulation as a definition for later reference.

	\begin{definition}[Occupation-measure primal problem]
		\label{def:measure_control}
		The occupation-measure reformulation of~\eqref{eq:path_primal} optimizes
		over the per-step state distributions
		$\rho_t\in\calP(\calX)$ and state-input distributions
		$\lambda_t\in\calP(\calX\times\calU)$, subject to the flow constraints
		from Proposition~\ref{prop:occupation_representation}:
\begingroup
\ifdim\linewidth>350pt
\def\occupationobjectivebreak{}
\else
\def\occupationobjectivebreak{\notag\\&\quad}
\fi
		\begin{align}
			\label{eq:occupation_primal}
			p^\star(\rho_0):=
			\inf_{\{\rho_t,\lambda_t\}}
			\quad &
			\sum_{t=0}^{T-1}
			\int_{\calX\times\calU}\ell(x,u)\,d\lambda_t(x,u)
			\occupationobjectivebreak
			+ \int_{\calX}\phi(x)\,d\rho_T(x)
			\notag\\
			\text{s.t.}\quad &
			\rho_0 \text{ prescribed},
			\notag\\
			& (\pi_X)_{\#}\lambda_t = \rho_t,
			\quad t=0,\dots,T-1,
			\notag\\
			& \rho_{t+1} = f_{\#}\lambda_t,
			\quad t=0,\dots,T-1.
		\end{align}
\endgroup
	\end{definition}

By Proposition~\ref{prop:occupation_representation}, the infimum in
\eqref{eq:path_primal} equals $p^\star(\rho_0)$ in
\eqref{eq:occupation_primal}: projection and reconstruction preserve every
feasible additive cost. The output map and output coordinates play no role
in this equivalence.

	\subsection{Strong Duality, Bellman Recursion, and Policy Extraction}

	We now derive the dual of~\eqref{eq:occupation_primal} and show that it
	coincides with Bellman's dynamic-programming recursion.

		Feasibility in the compact setting below follows from the system definition.
		Choose any $\bar u\in\calU$ and recursively set
		$\lambda_t=(\mathrm{id},\bar u)_{\#}\rho_t$ and
		$\rho_{t+1}=f_{\#}\lambda_t$ from the prescribed $\rho_0$.
		Since $f$ maps $\calX\times\calU$ into $\calX$, these measures satisfy the
		flow constraints. Compactness makes the product of the probability-measure
		spaces compact, continuity makes the flow constraints weakly closed, and
		the continuous costs are bounded. With additional expectation constraints,
		Corollary~\ref{thm:cduality} separately assumes feasibility and uses Slater's
		condition only for attainment of the constraint multipliers.

	\begin{theorem}[Strong duality and Bellman equivalence]
		\label{thm:strong_duality_bellman}
			Assume that $\calX$ and $\calU$ are nonempty and compact, that $f$ is continuous, and
		that $\ell$ and $\phi$ are continuous. Then, the dual of the
		occupation-measure problem~\eqref{eq:occupation_primal} is a maximization
		over value functions $V_t\in C(\calX)$, given by
		\begin{equation}
			\label{eq:dual_ocp}
			\begin{aligned}
				d^\star(\rho_0):=
				\sup_{\{V_t\}_{t=0}^T}\quad &
				\int_{\calX}V_0(x)\,d\rho_0(x)\\
				\text{s.t.}\quad &
				V_T(x)\le \phi(x),\quad x\in\calX,\\
				& V_t(x)\le \ell(x,u)+V_{t+1}(f(x,u)),\\
				& (x,u)\in\calX\times\calU,
				\; t=0,\dots,T-1.
			\end{aligned}
		\end{equation}
		Strong duality holds, i.e.,
		$p^\star(\rho_0)=d^\star(\rho_0)$.
		Moreover, the dual optimum is attained by the Bellman recursion, as follows: Setting
		$V_T^\star:=\phi$ and
		\begin{equation}
			\label{eq:bellman_recursion}
			V_t^\star(x)
			= \min_{u\in\calU}
			\bigl\{\ell(x,u)+V_{t+1}^\star(f(x,u))\bigr\},
			\quad t=T{-}1,\dots,0,
		\end{equation}
		yields a dual optimal family $\{V_t^\star\}_{t=0}^T$, and there exists a
		primal optimal pair $\{(\rho_t^\star,\lambda_t^\star)\}$ attaining the
		infimum in~\eqref{eq:occupation_primal}.
	\end{theorem}

		\begin{proof}
			\emph{Dual derivation.} Treat~\eqref{eq:occupation_primal} as a conic linear program over
			nonnegative finite measures. The fixed initial law and the flow equations
			enforce the unit masses of all feasible marginals.
			Introduce
			continuous multipliers $V_t\in C(\calX)$ for the marginal and dynamics
			constraints. With the sign convention used below, the Lagrangian is
\begingroup
\ifdim\linewidth>350pt
\def\lagrangiandisplaybreak{}
\else
\def\lagrangiandisplaybreak{\\&\quad}
\fi
\[
\begin{aligned}
			\mathcal L={}&
			\sum_{t=0}^{T-1}
			\int_{\calX\times\calU}
			\bigl[
			\ell(x,u)+V_{t+1}(f(x,u))-V_t(x)
			\bigr]\,d\lambda_t\lagrangiandisplaybreak
			+
			\int_{\calX}
			\bigl[
			\phi(x)-V_T(x)
			\bigr]\,d\rho_T
			+
			\int_{\calX}V_0(x)\,d\rho_0(x).
\end{aligned}
\]
\endgroup
			Minimizing this expression over nonnegative measures yields a finite
			lower bound exactly when
			\[
			V_T(x)\le \phi(x),\qquad x\in\calX,
			\]
			and
			\[
			\begin{aligned}
			V_t(x)&\le \ell(x,u)+V_{t+1}(f(x,u)),\\
			&\qquad (x,u)\in\calX\times\calU,\quad t=0,\ldots,T-1.
			\end{aligned}
			\]
			Under these inequalities, the infimum of the Lagrangian over
			$\lambda_t$ and $\rho_T$ is
			$\int_{\calX}V_0\,d\rho_0$, which gives the dual problem
			\eqref{eq:dual_ocp}.

			\emph{Weak duality.} Let $\{(\rho_t,\lambda_t)\}$ be primal feasible
			and $\{V_t\}$ dual feasible. Integrating the dual constraint
		$V_t(x)\le\ell(x,u)+V_{t+1}(f(x,u))$ against $\lambda_t$ gives
\begingroup
\ifdim\linewidth>350pt
\def\weakdualitydisplaybreak{}
\else
\def\weakdualitydisplaybreak{\\&\quad}
\fi
\[
\begin{aligned}
			\int V_t(x)\,d\lambda_t(x,u)
			&\le \int \ell(x,u)\,d\lambda_t(x,u)\weakdualitydisplaybreak
			+ \int V_{t+1}(f(x,u))\,d\lambda_t(x,u).
\end{aligned}
\]
		By the flow constraints~\eqref{eq:flow_marginal}
		and~\eqref{eq:flow_dynamics}, the left-hand side equals
		$\int V_t\,d\rho_t$ and the last term equals
		$\int V_{t+1}\,d\rho_{t+1}$. Summing over $t=0,\dots,T-1$, the
		$\int V_t\,d\rho_t$ terms telescope, leaving
\[
\begin{aligned}
			\int V_0\,d\rho_0
			&\le \sum_{t=0}^{T-1}\int \ell\,d\lambda_t
			+ \int V_T\,d\rho_T\weakdualitydisplaybreak
			\le \sum_{t=0}^{T-1}\int \ell\,d\lambda_t
			+ \int \phi\,d\rho_T,
\end{aligned}
\]
\endgroup
		where the second inequality uses $V_T\le\phi$.
		Taking the infimum over primal pairs and the supremum over dual families
		gives
		$d^\star(\rho_0)\le p^\star(\rho_0)$.

		\emph{Strong duality.} Set $V_T^\star:=\phi$ and compute
		$V_{T-1}^\star,\dots,V_0^\star$ backward by
		\eqref{eq:bellman_recursion}. By continuity of $f$ and $\ell$,
		compactness of $\calU$, and Berge's theorem~\cite{Berge1963}, the
		minimum is attained and each $V_t^\star$ is continuous. A measurable
		selection~\cite[Ch.~7]{BertsekasShreve1996} yields
		$\alpha_t^\star(x)\in\argmin_u\{\ell(x,u)+V_{t+1}^\star(f(x,u))\}$.
		Define $\rho_0^\star:=\rho_0$,
		$\lambda_t^\star:=(\mathrm{id},\alpha_t^\star)_{\#}\rho_t^\star$, and
		$\rho_{t+1}^\star:=(f(\cdot,\alpha_t^\star(\cdot)))_{\#}\rho_t^\star$.
		Then
		\[
		V_t^\star(x)
		=\ell(x,\alpha_t^\star(x))
		+V_{t+1}^\star(f(x,\alpha_t^\star(x))).
		\]
		Integrating against $\rho_t^\star$ and telescoping gives
		\[
		\int V_0^\star\,d\rho_0
		= \sum_{t=0}^{T-1}\int\ell\,d\lambda_t^\star
		+ \int\phi\,d\rho_T^\star.
		\]
		The left-hand side is a dual objective value and the right-hand side is
		a primal cost, so $d^\star(\rho_0) \geq p^\star(\rho_0)$. Together
		with weak duality, this yields $p^\star(\rho_0)=d^\star(\rho_0)$.
	\end{proof}

	\begin{remark}[Prior art, interpretation, and complementary roles]
		\label{rem:primal_dual_roles}
		The duality between occupation-measure linear programs and Bellman's
		recursion is classical in the MDP literature; see
		\cite{HernandezLermaLasserre1996,Altman1999} and
		\cite[Ch.~4]{Puterman1994}. Theorem~\ref{thm:strong_duality_bellman}
		records this classical duality in the
		behavioral-measure framework, where occupation measures expose
		distributional constraints while Bellman recursion yields value functions
		and measurable optimal policies.
	\end{remark}

	\begin{corollary}[Policy extraction from an optimal measure]
		\label{cor:policy_extraction}
		Under the assumptions of Theorem~\ref{thm:strong_duality_bellman}, let
		$\{V_t^\star\}$ be the optimal Bellman value functions and let
		$\{(\rho_t^\star,\lambda_t^\star)\}$ be any optimal solution of
		\eqref{eq:occupation_primal}. For each $t$, define the set of optimal
		actions at state $x$ as
		\[
		M_t(x)
		= \argmin_{u\in\calU}
		\bigl\{\ell(x,u)+V_{t+1}^\star(f(x,u))\bigr\}.
		\]
		Then the following hold.

		\textup{(a)}~\emph{Complementary slackness.} Disintegrate the optimal
		state-input measure as
		$\lambda_t^\star(dx,du)=\rho_t^\star(dx)\,\kappa_t^\star(du\mid x)$,
		where $\kappa_t^\star(\cdot\mid x)$ is the conditional distribution of
		the input given the state. Then, $\kappa_t^\star(\cdot\mid x)$ is
		supported on $M_t(x)$ for $\rho_t^\star$-almost every $x$. In other
		words, any optimal occupation measure concentrates its control actions
		on the Bellman-optimal set.

		\textup{(b)}~\emph{Deterministic optimal policy.} There exists a
		measurable selector
		$\alpha_t^\star:\calX\to\calU$ with
		$\alpha_t^\star(x)\in M_t(x)$ for every $x\in\calX$. The
		deterministic policy $\alpha_t^\star$ generates feasible and optimal
		state and occupation sequences for~\eqref{eq:occupation_primal} via
		\[
		\begin{aligned}
			\bar\rho_0&:=\rho_0,\qquad
			\bar\rho_{t+1}
			:=\bigl(f(\cdot,\alpha_t^\star(\cdot))\bigr)_{\#}\bar\rho_t,\\
			\bar\lambda_t&:=(\mathrm{id},\alpha_t^\star)_{\#}\bar\rho_t.
		\end{aligned}
		\]
	\end{corollary}

	\begin{proof}
		\emph{Part~(a).} Define the Bellman slack at time $t$ as
		\[
		g_t(x,u):=\ell(x,u)+V_{t+1}^\star(f(x,u))-V_t^\star(x).
		\]
		By definition of the Bellman recursion~\eqref{eq:bellman_recursion},
		$g_t(x,u)\ge 0$ for all $(x,u)$, with equality precisely when
		$u\in M_t(x)$. Repeating the telescoping argument from the proof of
		Theorem~\ref{thm:strong_duality_bellman}, now with the optimal pair
		$\{(\rho_t^\star,\lambda_t^\star)\}$, gives
		\[
		0
		= \sum_{t=0}^{T-1}
		\int g_t\,d\lambda_t^\star
		+ \int\bigl(\phi-V_T^\star\bigr)\,d\rho_T^\star.
		\]
		Since every term is nonnegative, $\int g_t\,d\lambda_t^\star=0$ for every
		$t$. Disintegrating
		$\lambda_t^\star(dx,du)=\rho_t^\star(dx)\kappa_t^\star(du\mid x)$ and
		using $g_t\ge0$ gives
		$\int_{\calU} g_t(x,u)\,\kappa_t^\star(du\mid x)=0$ for
		$\rho_t^\star$-almost every $x$, so
		$\kappa_t^\star(\cdot\mid x)$ is supported on $M_t(x)$.

		\emph{Part~(b).} The set-valued map $M_t$ has measurable graph and
		nonempty compact values (nonemptiness follows from continuity of the
		integrand and compactness of $\calU$). The measurable selection
		theorem \cite[Ch.~7]{BertsekasShreve1996} provides a measurable
		function $\alpha_t^\star:\calX\to\calU$ with
		$\alpha_t^\star(x)\in M_t(x)$ for every $x\in\calX$.

		Feasibility is immediate from the definitions. Since
		$\alpha_t^\star(x)\in M_t(x)$,
		\[
		V_t^\star(x)
		=\ell(x,\alpha_t^\star(x))
		+V_{t+1}^\star(f(x,\alpha_t^\star(x)))
		\quad\text{for every } x\in\calX.
		\]
		Integrating against $\bar\rho_t$ and telescoping gives
		\[
		\int V_0^\star\,d\rho_0
		= \sum_{t=0}^{T-1}\int\ell\,d\bar\lambda_t
		+ \int\phi\,d\bar\rho_T.
		\]
		The left-hand side equals $d^\star(\rho_0)=p^\star(\rho_0)$ by strong
		duality, so the deterministic pair achieves the optimal cost.
	\end{proof}

	Conditional input randomization does not improve the optimum of the
unconstrained additive problem~\eqref{eq:occupation_primal}. The
	deterministic selector attains an optimal occupation sequence, but need not
	reproduce the full temporal coupling of another optimal behavioral measure.

\subsection{Constrained Duality on the Full Trajectory Law}
\label{subsec:full_path_duality}

Continuous functions of the entire trajectory include nonadjacent cross-time
statistics and nonadditive costs. Their expectations remain affine in the
trajectory law. The following result combines the classical
measure-linear-programming and minimax viewpoint
\cite{BhattBorkar1996,KurtzStockbridge1998,Sion1958} to the initial-law slice.

\begin{theorem}[Full-path constrained duality]
\label{thm:path_duality}
Let Assumption~\ref{ass:standing} hold with $\calX$ and $\calU$ compact, and
fix $\rho_0\in\calP(\calX)$. Let $c_0,\ldots,c_m$ be continuous real-valued
functions on the compact realized behavior $\mathfrak B_T$, and let $b\in\R^m$. Suppose that
\begin{equation}
\label{eq:path_constrained_primal}
\begin{aligned}
p_{\mathrm{path}}(\rho_0):=
\inf_{\mu\in\calM_\calB(\rho_0)}\quad &\int c_0\,d\mu\\
\text{subject to}\quad &\int c_j\,d\mu\le b_j,\quad j=1,\ldots,m
\end{aligned}
\end{equation}
is feasible. For $\eta\in\R_+^m$, define
\begin{equation}
\label{eq:path_lagrangian_value}
V^\eta(x):=\min_{u\in\calU^T}
\Bigl(c_0+\sum_{j=1}^m\eta_jc_j\Bigr)(\Phi_T(x,u)).
\end{equation}
Then the constrained primal attains its minimum, $V^\eta\in C(\calX)$, and
\begin{equation}
\label{eq:path_constrained_dual}
p_{\mathrm{path}}(\rho_0)
=\sup_{\eta\ge0}\Bigl[\int V^\eta\,d\rho_0-\eta^\top b\Bigr].
\end{equation}
For each fixed $\eta$, a Borel input schedule $g^\eta:\calX\to\calU^T$
attains the minimum in~\eqref{eq:path_lagrangian_value} for every $x$.
If a law in $\calM_\calB(\rho_0)$ satisfies all constraints strictly, the
supremum in~\eqref{eq:path_constrained_dual} is attained. Every optimal pair
$(\mu^\star,\eta^\star)$ then satisfies
\[
\eta_j^\star\Bigl(\int c_j\,d\mu^\star-b_j\Bigr)=0,
\quad j=1,\ldots,m.
\]
\end{theorem}

\begin{proof}
The continuous image
$\mathfrak B_T=\Phi_T(\calX\times\calU^T)$ is compact, even when $\calY$
is not compact. Corollary~\ref{cor:initial_slice} makes
$F:=\calM_\calB(\rho_0)$ nonempty, compact, and convex. Set
$J(\mu)=\int c_0\,d\mu$ and $C_j(\mu)=\int c_j\,d\mu$.
The intrinsic weak topology on $\calP(\mathfrak B_T)$ agrees with its
subspace topology in $\calP(\Omega_T)$. These functionals are therefore
affine and weakly continuous. The feasible subset of $F$ is compact, so the
primal minimum is attained and finite.

Define
$L(\mu,\eta):=J(\mu)+\sum_j\eta_j(C_j(\mu)-b_j)$.
It is continuous and affine in either variable. Sion's theorem
\cite{Sion1958}, with compact convex domain $F$ and convex domain
$\R_+^m$, gives
\[
\inf_{\mu\in F}\sup_{\eta\ge0}L(\mu,\eta)
=\sup_{\eta\ge0}\inf_{\mu\in F}L(\mu,\eta).
\]
The left side is the constrained primal value: the inner supremum is
$J(\mu)$ for feasible $\mu$, and is $+\infty$ otherwise.

For fixed $\eta$, the function
$(x,u)\mapsto(c_0+\sum_j\eta_jc_j)(\Phi_T(x,u))$ is continuous on the
compact product $\calX\times\calU^T$. Berge's maximum theorem
\cite{Berge1963} gives continuity of $V^\eta$, and measurable selection
\cite[Ch.~7]{BertsekasShreve1996} gives a Borel minimizing schedule
$g^\eta$. Every law $\pi\in\Pi(\rho_0)$ has expected Lagrangian path cost
at least $\int V^\eta\,d\rho_0$. The law
$\pi^\eta=(\id,g^\eta)_{\#}\rho_0$ attains this lower bound. By
Corollary~\ref{cor:homeo}, its trajectory law belongs to $F$, and therefore
\[
\inf_{\mu\in F}L(\mu,\eta)
=\int V^\eta\,d\rho_0-\eta^\top b.
\]
Substitution into the minimax identity proves
\eqref{eq:path_constrained_dual}.

For multiplier attainment, let $\mu_S$ be strictly feasible and choose
$s>0$ such that $C_j(\mu_S)-b_j\le-s$ for every $j$. The dual function
$q(\eta):=\inf_{\mu\in F}L(\mu,\eta)$ satisfies
\[
q(\eta)\le J(\mu_S)-s\|\eta\|_1.
\]
Since continuous costs are bounded on the compact path set, $q(0)$ is
finite. No maximizer is lost by restricting to
$\|\eta\|_1\le(J(\mu_S)-q(0)+1)/s$.
The function $q$ is upper semicontinuous as an infimum of continuous affine
functions of $\eta$, so it attains its supremum on this compact set.
For any optimal pair, weak duality and primal feasibility give
\[
p_{\mathrm{path}}=q(\eta^\star)
\le L(\mu^\star,\eta^\star)
\le J(\mu^\star)=p_{\mathrm{path}}.
\]
Thus the sum of the nonpositive complementary-slackness terms is zero,
which forces every term to vanish.
\end{proof}

For each priced problem, the law generated by $g^\eta$ is an extreme point
of the initial-law slice by Proposition~\ref{prop:slice}. The constrained
primal may require randomization among deterministic schedules. If an
optimal multiplier exists, every primal optimizer minimizes that priced
problem; without multiplier attainment, no finite-price extraction is
asserted. The value function
$V^\eta$ minimizes a full input schedule conditional on the observed initial
state. For nonadditive path functions, it need not satisfy a Bellman
recursion on the original state alone; retaining the history or augmenting
the state restores a dynamic-programming description.

A variance or covariance constraint belongs to this linear formulation only
when it can be written using fixed path functions. Prescribing the relevant
means makes a covariance constraint an expectation constraint; a raw cross
moment such as $\E[X_sX_t]$ of scalar coordinates is affine without
prescribing means.

\begin{corollary}[Constrained duality on a compact path restriction]
\label{cor:compact_path_duality}
Let $K$ be a nonempty compact subset of a Polish space, let
$c_0,\ldots,c_m\in C(K)$, and let $b\in\R^m$. With no separate marginal
constraint, suppose the problem
\[
\begin{aligned}
p_K:=\inf_{\mu\in\calP(K)}\quad &\int c_0\,d\mu\\
\text{subject to}\quad &\int c_j\,d\mu\le b_j,\quad j=1,\ldots,m
\end{aligned}
\]
is feasible. It attains its minimum and
\begin{equation}
\label{eq:compact_path_dual}
p_K=\sup_{\eta\ge0}\left[\min_{\omega\in K}
\left(c_0+\sum_{j=1}^m\eta_jc_j\right)(\omega)-\eta^\top b\right].
\end{equation}
Strict feasibility ensures multiplier attainment and complementary slackness.
\end{corollary}

\begin{proof}
Apply the compact-law minimax argument of Theorem~\ref{thm:path_duality}
to $\calP(K)$. For fixed prices, a Dirac law at a pathwise minimizer
attains the inner infimum. The same strict-feasibility bound proves
multiplier attainment when applicable.
\end{proof}

Continuous expectation
equalities can be represented by paired inequalities; these pairs preclude
the strict-feasibility hypothesis used for multiplier attainment, but leave
the zero-gap conclusion intact.

\subsection{Additive Constraints as an Occupation Projection}

When the objective and each constraint are sums of stage and terminal
functions, the occupation marginals preserve every quantity being optimized
or constrained. The full-path theorem then recovers classical constrained
occupation duality.

\begin{corollary}[Constrained occupation duality]
\label{thm:cduality}
Assume the setting of Theorem~\ref{thm:strong_duality_bellman}. For
$j=1,\ldots,m$, let $c_{j,t}\in C(\calX\times\calU)$ for $t<T$ and
$c_{j,T}\in C(\calX)$, and let $b\in\R^m$. Let $p_c(\rho_0)$ be the occupation problem
\eqref{eq:occupation_primal} with additional constraints
\begin{equation}
\label{eq:cprimal}
\sum_{t=0}^{T-1}\int c_{j,t}\,d\lambda_t
+\int c_{j,T}\,d\rho_T\le b_j,
\quad j=1,\ldots,m.
\end{equation}
For $\eta\ge0$, let $V_t^\eta$ be the Bellman recursion with stage costs
$\ell+\sum_j\eta_jc_{j,t}$ and terminal cost
$\phi+\sum_j\eta_jc_{j,T}$. If the constrained problem is feasible, it
attains its minimum and
\begin{equation}
\label{eq:cdual}
p_c(\rho_0)=\sup_{\eta\ge0}
\Bigl[\int V_0^\eta\,d\rho_0-\eta^\top b\Bigr].
\end{equation}
Strict feasibility ensures multiplier attainment and complementary slackness.
\end{corollary}

\begin{proof}
Apply Theorem~\ref{thm:path_duality} with
$c_0=\sum_{t<T}\ell(X_t,U_t)+\phi(X_T)$ and
$c_j=\sum_{t<T}c_{j,t}(X_t,U_t)+c_{j,T}(X_T)$.
Proposition~\ref{prop:occupation_representation} projects and reconstructs
laws with all these integrals unchanged, so the path and occupation
problems have the same feasible cost values. Because the dynamics are
deterministic, optimizing an open-loop schedule from each fixed initial
state has the same value as feedback dynamic programming. For each
$\eta$, additive separability therefore identifies the pathwise value $V^\eta$ with the Bellman value
$V_0^\eta$ in Theorem~\ref{thm:strong_duality_bellman}. The remaining claims
follow from Theorem~\ref{thm:path_duality}.
\end{proof}

	\begin{remark}[Randomization under active constraints]
		\label{rem:randomization}
		With active expectation constraints, optimal behavioral measures may
		require randomization, unlike the unconstrained setting of
		Corollary~\ref{cor:policy_extraction}(b). In constrained Markov decision
		processes, optima can be chosen as mixtures of finitely many deterministic
		policies~\cite{Altman1999,FeinbergShwartz2002}.
	\end{remark}

		\section{Compactness, LTI Specialization, and Stochastic Extension}
		\label{sec:structural}
		\label{sec:compactness_lti_stochastic}

	This section treats compactness and existence, then the measure-level
	Fundamental Lemma for controllable LTI systems, and finally a stochastic
	extension based on history-conditional kernel consistency.

	\subsection{Compactness and Existence}

	In this subsection, $\calX\subseteq\R^{n_x}$, $\calU\subseteq\R^{n_u}$,
	and $\calY\subseteq\R^{n_y}$ are closed. The next lemma gives two standard
	compactness regimes: compact signal spaces, or a uniform second-moment
	bound.

		\begin{lemma}[Tightness and compactness]
			\label{lem:tightness}
			Let $\mathcal F\subseteq\mathcal M_\mathcal B$ be a family. Suppose
that one of the following conditions holds.
			\textup{(i)} The spaces $\calX$, $\calU$, and $\calY$ are compact.

\textup{(ii)} There exists $C>0$ such that $\mathcal F$ is contained in the
					moment-bounded set
					$\mathcal{M}_{\mathcal{B}}^{(2)}(C):=
					\{\mu\in\mathcal{M}_{\mathcal{B}}:\Psi(\mu)\le C\}$, where
					\begin{equation}
						\label{eq:moment_bound}
\ifdim\linewidth>350pt
\def\momentbounddisplaybreak{}
\else
\def\momentbounddisplaybreak{\\&\quad}
\fi
						\begin{aligned}
							\Psi(\mu)
							&:=
							\sum_{t=0}^{T-1}
							\int_{\Omega_T}
							\bigl(\norm{X_t}^2+\norm{U_t}^2+\norm{Y_t}^2\bigr)\,d\mu
\momentbounddisplaybreak + \int_{\Omega_T}\norm{X_T}^2\,d\mu.
						\end{aligned}
					\end{equation}

			Then, under~\textup{(i)}, $\mathcal{M}_{\mathcal{B}}$ is compact in
			$\calP(\Omega_T)$. Under~\textup{(ii)},
			$\mathcal{M}_{\mathcal{B}}^{(2)}(C)$ is tight and compact in
			$\calP(\Omega_T)$, and any family contained in
			$\mathcal{M}_{\mathcal{B}}^{(2)}(C)$ is relatively compact.
		\end{lemma}

	\begin{proof}
		Case~\textup{(i)} is immediate: if $\calX$, $\calU$, and $\calY$ are
		compact, then so is $\Omega_T$, hence $\calP(\Omega_T)$ is compact and
		$\mathcal{M}_{\mathcal{B}}$ is compact by
		Lemma~\ref{lem:closed_behavior}.

		For case~\textup{(ii)}, Markov's inequality gives, for every $R>0$,
		\[
		\mu\!\left(
		\sum_{t=0}^{T-1}\bigl(\norm{X_t}^2+\norm{U_t}^2+\norm{Y_t}^2\bigr)
		+ \norm{X_T}^2 > R^2
		\right)
		\le \frac{C}{R^2}.
		\]
		Choosing $R\ge \sqrt{C/\varepsilon}$ yields the compact sublevel set
		\[
\ifdim\linewidth>350pt
K_{R}:=\Bigl\{\omega\in\Omega_{T}:
\sum_{t=0}^{T-1}\bigl(\norm{X_{t}(\omega)}^{2}+\norm{U_{t}(\omega)}^{2}
+\norm{Y_{t}(\omega)}^{2}\bigr)+\norm{X_{T}(\omega)}^{2}\le R^{2}\Bigr\}
\else
		K_{R}:=\Big\{\omega\in\Omega_{T}:\begin{aligned}[t]\sum_{t=0}^{T-1}\bigl( & \norm{X_{t}(\omega)}^{2}+\norm{U_{t}(\omega)}^{2}\\
			& \left.+\norm{Y_{t}(\omega)}^{2}\bigr)+\norm{X_{T}(\omega)}^{2}\le R^{2}\right\}
		\end{aligned}
\fi
		\]
		Because $\calX$, $\calU$, and $\calY$ are closed subsets of
		finite-dimensional Euclidean spaces, $K_R$ is closed and bounded in
		$\Omega_T$, hence compact.
		Moreover, $\mu(K_R)\ge 1-\varepsilon$ for every
		$\mu\in\mathcal{M}_{\mathcal{B}}^{(2)}(C)$, proving tightness; relative
		compactness follows from Prokhorov's theorem
		\cite[Thm.~5.1]{Billingsley1999}. Since the integrand in
		\eqref{eq:moment_bound} is lower semicontinuous and bounded below, the
		functional $\Psi$ is weakly lower semicontinuous. Combined with weak
		closedness of $\mathcal{M}_{\mathcal{B}}$ from
		Lemma~\ref{lem:closed_behavior}, this implies that
		$\mathcal{M}_{\mathcal{B}}^{(2)}(C)$ is weakly closed and therefore
			compact.
	\end{proof}

		\begin{theorem}[Existence of an optimal occupation solution]
			\label{thm:existence_optimal}
			Assume that $f$ is continuous, that the stage cost $\ell$ and terminal
			cost $\phi$ are lower semicontinuous and bounded below, and that the
			feasible set of the occupation-measure
			problem~\eqref{eq:occupation_primal} is nonempty and admits a finite-value
minimizing sequence that is tight in the product space
			$\prod_{t=0}^{T-1}\calP(\calX\times\calU)\times
		\prod_{t=1}^{T}\calP(\calX)$. Then, \eqref{eq:occupation_primal} admits an
		optimal solution.
	\end{theorem}

	\begin{proof}
		The finite-value minimizing sequence and bounded-below costs make the
infimum finite. Choose the tight minimizing sequence
$\{(\rho_t^k,\lambda_t^k)\}_{k\ge1}$ in the hypothesis. By
		tightness and Prokhorov's theorem~\cite[Thm.~5.1]{Billingsley1999}, a
		subsequence converges weakly to $\rho_t$ and $\lambda_t$ for every $t$.
		Passing to the limit in $(\pi_X)_{\#}\lambda_t^k=\rho_t^k$ yields
		$(\pi_X)_{\#}\lambda_t=\rho_t$, and passing to the limit in
		$\rho_{t+1}^k=f_{\#}\lambda_t^k$ yields $\rho_{t+1}=f_{\#}\lambda_t$
		because $\varphi\circ f\in C_b(\calX\times\calU)$ for every
		$\varphi\in C_b(\calX)$. Thus the limit tuple is feasible. Since $\ell$
		and $\phi$ are lower semicontinuous and bounded below, the objective is
		weakly lower semicontinuous, so the limit attains the infimum.
	\end{proof}

The tightness hypothesis may be checked on a finite objective sublevel
containing a minimizing sequence, rather than on every feasible law. In
particular, tightness of the whole feasible set and one finite feasible
value suffice. This distinction matters for unbounded controls: with
$\calX=\{0\}$, $\calU=\R$, $f=0$, $\ell(0,u)=u^2$, and $\phi=0$, each
finite objective sublevel is tight by its uniform second-moment bound,
whereas the full feasible family contains every input law $\delta_u$ and
is not tight. The theorem applies to the sublevels and gives the optimum
$\delta_0$.

		\begin{corollary}[Tightness of the occupation feasible set]
			\label{cor:tightness_occupation_feasible}
			Assume either that $\calX$, $\calU$, and $\calY$ are compact, or that
			the feasible behavioral measures in
			$\mathcal{M}_{\mathcal{B}}(\rho_0)$ are contained in
			$\mathcal{M}_{\mathcal{B}}^{(2)}(C)$ for some $C>0$. Then the induced
			feasible set of occupation tuples in~\eqref{eq:occupation_primal} is
			tight in
			\[
			\prod_{t=0}^{T-1}\calP(\calX\times\calU)\times
			\prod_{t=1}^{T}\calP(\calX).
			\]
			Consequently, if this feasible set is nonempty, admits at least one
			feasible tuple of finite objective value, and $\ell$ and $\phi$ are
			lower semicontinuous and bounded below, then
			\eqref{eq:occupation_primal} admits an optimal solution.
		\end{corollary}

		\begin{proof}
			Lemma~\ref{lem:tightness} gives tightness of the relevant family of
			behavioral measures. Since the coordinate projections
			$(X_t,U_t):\Omega_T\to\calX\times\calU$ and $X_t:\Omega_T\to\calX$ are
			continuous, the induced marginals $\lambda_t$ and $\rho_t$ are tight
			for each $t$. The product family is therefore tight, and
			Proposition~\ref{prop:occupation_representation} identifies it with the
			feasible set of \eqref{eq:occupation_primal}. The claim follows from
			Theorem~\ref{thm:existence_optimal}.
		\end{proof}

\subsection{LTI Specialization and the Fundamental-Lemma Bridge}
\label{subsec:lti_fundamental_lemma}

We apply the closed-behavior formulation of
Subsection~\ref{subsec:closed_behaviors} directly to external trajectories.
Here $L$ replaces $T$ as the trajectory-window length, and $N$ denotes the
length of the data record. Set $w_t=(u_t,y_t)\in\R^{n_u+n_y}$ and
$\calW_L=(\R^{n_u+n_y})^L$. Consider the controllable LTI system
\[
x_{t+1}=Ax_t+Bu_t,\qquad y_t=Cx_t+Du_t,
\quad t=0,\ldots,L-1,
\]
with free initial state. Its external behavior $\calB_L\subseteq\calW_L$
is the image of the linear map from initial states and input schedules to
external trajectories, hence a closed linear subspace
\cite{Willems1991,PoldermanWillems1998}. Define
\[
\mathcal M_\calB^L:=\calM(\calB_L)=\calP(\calB_L).
\]
Lemma~\ref{lem:closed_behavior} gives its convexity, weak closedness,
Dirac extreme points, and density of finite mixtures.

The projection that forgets state coordinates maps the realized path set
onto $\calB_L$. In this finite-dimensional linear setting it has a continuous
linear right inverse on its image, so Lemma~\ref{lem:param} lifts every
external law to a realized path law. This lift need not be unique. Thus the
external description is intrinsic to $\calB_L$, while the realized
specialization of Section~\ref{sec:behavioral_measures} and the control
problems of Section~\ref{sec:optimal_control} retain state coordinates.

	The next theorem lifts the classical
	Fundamental Lemma of Willems~\cite{Willems2005Fundamental} from individual
	trajectories to probability measures on trajectories: under persistency of
	excitation, the entire behavioral-measure set can be generated by choosing a
	probability distribution over the coefficient vector in the data Hankel
	matrix built from a noise-free exact data trajectory
	$w^d = (w_0^d, \ldots, w_{N-1}^d)$
	of length $N$ as
	\[
	H_L(w^d) := \begin{bmatrix}
		w_0^d & w_1^d & \cdots & w_{N-L}^d \\
		w_1^d & w_2^d & \cdots & w_{N-L+1}^d \\
		\vdots & \vdots & \ddots & \vdots \\
		w_{L-1}^d & w_L^d & \cdots & w_{N-1}^d
	\end{bmatrix}.
	\]
	Each column of $H_L(w^d)$ is a consecutive length-$L$ window of the
	input/output data trajectory
	$w^d = (w_0^d, \ldots, w_{N-1}^d)$, with
	$w_t^d = (u_t^d, y_t^d) \in \R^{n_u+n_y}$, shifted by one time step.

	\begin{definition}[Persistency of excitation]
		\label{def:pe}
		An input sequence $u^d=(u^d_0,\dots,u^d_{N-1})$ with $u^d_t\in\R^{n_u}$ is
		\emph{persistently exciting of order $k$} if its depth-$k$ Hankel matrix
		\[
		H_k(u^d):=\begin{bmatrix}
			u^d_0 & u^d_1 & \cdots & u^d_{N-k}\\
			\vdots & \vdots & & \vdots\\
			u^d_{k-1} & u^d_k & \cdots & u^d_{N-1}
		\end{bmatrix}
		\]
		has full row rank $k\,n_u$.
	\end{definition}

	\begin{fact}[Fundamental Lemma~\cite{Willems2005Fundamental,vanWaarde2020Fundamental}]
		\label{fact:fl}
		Let the controllable LTI system above have a minimal realization of order
		$n$, and let $w^d=(u^d,y^d)$ be a length-$N$ trajectory it generates. If
		$u^d$ is persistently exciting of order $L+n$, then
		$\col H_L(w^d)=\calB_L$.
	\end{fact}

	The behavior has dimension $\dim\calB_L=L\,n_u+n$ whenever the window length
	$L$ is at least the lag, namely the observability index of a minimal
realization of the external behavior~\cite{PoldermanWillems1998,MarkovskyDorfler2021}, a
	condition met by the data-driven study of Section~\ref{sec:numerics}.

	Theorem~\ref{thm:measure_level_FL} is a direct law-level consequence of this
	classical spanning statement, obtained by pushforward and a linear right
	inverse.

	\begin{theorem}[Measure-Level Fundamental Lemma]
		\label{thm:measure_level_FL}
		Let $w^d=(u^d,y^d)$ be a noise-free external trajectory of length $N$
		generated exactly by the controllable LTI system above. Let $n$ denote the McMillan degree of
		the system, and suppose that the input component $u^d$ is persistently
		exciting of order $L+n$. More generally, persistency of excitation of
		order $L+\bar n$ is sufficient for any known upper bound
		$\bar n\ge n$; in particular, a state dimension $n_x$ of a realization is
		a valid upper bound. Let
		$H_L(w^d)\in\R^{L(n_u+n_y)\times(N-L+1)}$ denote the Hankel matrix built
		from $w^d$. Then the behavioral-measure set on external trajectories
		satisfies
		\begin{equation}
			\label{eq:measure_level_FL}
			\mathcal{M}_{\mathcal{B}}^{L}
			= \bigl\{\pushfwd{(H_L(w^d))}\nu :
			\nu\in\calP(\R^{N-L+1})\bigr\}.
		\end{equation}
		In other words, every probability measure supported on the behavior
		$\calB_L$ can be generated by choosing a probability distribution
		$\nu$ on the coefficient space $\R^{N-L+1}$ and pushing it forward
		through the Hankel matrix, and every such pushforward produces a
		valid behavioral measure.
	\end{theorem}

	\begin{proof}
		By Fact~\ref{fact:fl}, under the noise-free data and
		persistency-of-excitation hypotheses $\calB_L = \col H_L(w^d)$, so the
		linear map $H_L:\R^{N-L+1}\to\calW_L$ has image
		$\mathrm{Im}\,H_L=\calB_L$. Apply
		Lemma~\ref{lem:param} with $P=\R^{N-L+1}$, $S=\calW_L$, $\Phi=H_L$,
		$E=\calB_L$, and $\sigma=H_L^\dagger$, the Moore--Penrose pseudoinverse.
		The behavior $\calB_L$ is a closed linear subspace of $\calW_L$, and
		$H_L^\dagger$ is a continuous linear right inverse on $\calB_L$: the
		product $H_L H_L^\dagger$ is the orthogonal projector onto
		$\mathrm{Im}\,H_L=\calB_L$, so $H_L H_L^\dagger w = w$ for every
		$w\in\calB_L$. Lemma~\ref{lem:param} then yields
		\[
		\mathcal{M}_{\mathcal{B}}^{L}=\calP(\calB_L)
		=\bigl\{\pushfwd{(H_L)}\nu:\nu\in\calP(\R^{N-L+1})\bigr\},
		\]
		which is the set equality~\eqref{eq:measure_level_FL}, and it identifies
		the canonical lift $\nu=\pushfwd{(H_L^\dagger)}\mu$ supported on the row
		space of $H_L$.
	\end{proof}

	The stochastic Fundamental Lemma of Pan, Ou, and
Faulwasser~\cite{PanOuFaulwasser2023}
	represents square-integrable random trajectories by a deterministic
	realization-data Hankel matrix and a square-integrable random coefficient,
under controllability of the combined control/disturbance pair and
persistency of excitation of the corresponding combined realization record.
	For deterministic dynamics, remove the disturbance channel and impose
	controllability and input persistency of excitation for the resulting
	system. This gives the random-variable counterpart of
	Theorem~\ref{thm:measure_level_FL}. If $G$ has law $\nu$, then $H_LG$ has
	law $(H_L)_{\#}\nu$; conversely, $G=H_L^\dagger W$ is square-integrable
	whenever $W$ is. The present statement covers all Borel probability
	laws, does not fix a common probability space, and singles out the
	pseudoinverse lift. These are differences of formulation and scope.

	\begin{remark}[Non-uniqueness of the lift]
		\label{rem:lift_nonuniqueness}
		The coefficient-space distribution $\nu$ producing a given behavioral
		measure $\mu$ need not be unique. Whenever $N-L+1 > \dim(\calB_L)$, the
		Hankel matrix has a nontrivial kernel, so different distributions on
		coefficient space can produce the same trajectory distribution. The
		canonical choice $\nu = \pushfwd{(H_L^\dagger)}\mu$ used in the proof
		is distinguished by being the unique lift supported on the row space
		of $H_L$.
	\end{remark}

	\begin{remark}[Recovering the classical Fundamental Lemma]
		\label{rem:measure_FL_interpretation}
		Write $H=H_L(w^d)$. For finitely many trajectories $w_i=Hg_i$ and
		probabilities $p_i\ge0$ with $\sum_i p_i=1$, pushforward gives
		\[
		H_{\#}\Bigl(\sum_i p_i\delta_{g_i}\Bigr)
		=\sum_i p_i\delta_{Hg_i},\qquad
		\E[w]=H\sum_i p_i g_i.
		\]
		The $p_i$ are mixture probabilities, while the entries of each $g_i$ are
		arbitrary real coefficients in a linear combination of Hankel columns.
		The trajectory law $\sum_i p_i\delta_{Hg_i}$ generally differs from the
		Dirac law $\delta_{H\sum_i p_i g_i}$ at its mean. The factorization therefore
		does not restrict trajectories to the convex hull of the recorded columns.
		When $\nu = \delta_g$ is a Dirac mass on a single coefficient vector,
		the pushforward $\pushfwd{(H_L)}\delta_g = \delta_{H_Lg}$ is a Dirac
		mass on the trajectory $w = H_L(w^d)g$. This recovers the classical
		Fundamental Lemma: every admissible trajectory $w \in \calB_L$ can be
		written as a linear combination of columns of the Hankel matrix. The
		law-level factorization is equivalent to the classical spanning statement:
		spanning gives the factorization by pushforward through a right inverse,
		and applying the factorization to Dirac laws gives spanning. Its use here
		is to transfer distributional specifications to coefficient space.
	\end{remark}

	The measure-level Fundamental Lemma characterizes the full distributional
	structure of the behavior. A natural question is what remains when only
	first-order statistics are extracted. The following lemma shows that the
	set of mean trajectory vectors under all behavioral measures is exactly
	the classical behavior $\calB_L$.

		\begin{lemma}[Degree-one moment characterization]
			\label{lem:degree_one_moments}
		Let $\calP_1(\calB_L)$ denote the set of behavioral measures with
		finite first moment,
		\[
		\calP_1(\calB_L)
		:=
		\left\{
		\mu\in\calP(\calB_L) :
		\int_{\calB_L}\norm{w}\,d\mu(w)<\infty
		\right\},
		\]
		and for each such measure define the degree-one moment vector, which
		stacks a normalizing entry with the mean trajectory,
		\[
		m_1(\mu)
		:=
		\begin{bmatrix}
			1\\[1mm]
			\int_{\calB_L} w\,d\mu(w)
		\end{bmatrix},
		\qquad
		\mu\in\calP_1(\calB_L).
		\]
		Then, the set of all degree-one moment vectors coincides with the
		set of admissible trajectories (augmented by a leading one):
		\[
		\left\{m_1(\mu):\mu\in\calP_1(\calB_L)\right\}
		=
		\left\{
		\begin{bmatrix}
			1\\
			w
		\end{bmatrix}
		: w\in\calB_L
		\right\}.
		\]
\end{lemma}

	\begin{proof}
		We show both inclusions.

			\emph{Every mean trajectory lies in $\calB_L$.} Let
			$\mu\in\calP_1(\calB_L)$. Since $\calB_L$ is a closed linear
			subspace of the finite-dimensional space $\calW_L$, there is a matrix
			$A$ such that $\calB_L=\ker A$. The mean
			$\bar{w} := \int_{\calB_L} w\,d\mu(w)$ is well defined, and
			$Aw=0$ $\mu$-almost surely. Hence
			\[
			A\bar w
			=
			A\int w\,d\mu(w)
			=
			\int Aw\,d\mu(w)
			=0,
			\]
			so $\bar w\in\calB_L$.
			Hence, $m_1(\mu) = [1;\,\bar{w}]$ with $\bar{w}\in\calB_L$.

		\emph{Every trajectory is a mean.} Conversely, for any
		$w\in\calB_L$, the Dirac measure $\delta_w$ belongs to
		$\calP_1(\calB_L)$ and satisfies
		\[
		m_1(\delta_w)=
		\begin{bmatrix}
			1\\
			w
		\end{bmatrix},
		\]
		so every element of $\calB_L$ is realized as the mean of some
		behavioral measure. This establishes the set equality.
	\end{proof}

		Combining the degree-one moment characterization with the measure-level
		Fundamental Lemma yields the Hankel characterization of the moment set.

	\begin{corollary}[Degree-one moment/Hankel bridge]
		\label{cor:fundamental_bridge}
		Under the assumptions of Theorem~\ref{thm:measure_level_FL}, the
		degree-one moment set equals the column space of
		the Hankel matrix (augmented by a leading one):
		\begin{equation}
			\label{eq:affine_hull_hankel}
\ifdim\linewidth>350pt
\def\momentbridgedisplaybreak{}
\else
\def\momentbridgedisplaybreak{\\&\quad}
\fi
\begin{aligned}
&\bigl\{m_1(\mu):\mu\in\calP_1(\calB_L)\bigr\}\momentbridgedisplaybreak
			= \biggl\{
			\begin{bmatrix}
				1\\
				H_L(w^d)g
			\end{bmatrix}
			: g\in\R^{N-L+1}
			\biggr\}
			= \{1\}\times \col H_L(w^d).
\end{aligned}
		\end{equation}
	\end{corollary}

	\begin{proof}
		The result follows by combining the measure-level Fundamental
			Lemma (Theorem~\ref{thm:measure_level_FL}) with the degree-one
			moment characterization (Lemma~\ref{lem:degree_one_moments}).

		For any $\mu\in\calP_1(\calB_L)$, Theorem~\ref{thm:measure_level_FL}
		gives $\mu = \pushfwd{(H_L)}\nu$ for the canonical lift
		$\nu = \pushfwd{(H_L^\dagger)}\mu$. The mean coefficient vector
		$\bar{g} = \int g\,d\nu(g)$ is well defined because
		\[
		\int\|g\|\,d\nu
		= \int\|H_L^\dagger w\|\,d\mu
		\le \|H_L^\dagger\|\int\|w\|\,d\mu
		< \infty.
		\]
		The mean trajectory under $\mu$ is then $\bar{w} = H_L\bar{g}$, so
		$m_1(\mu) = [1;\,H_L\bar{g}]$, which belongs to
		$\{1\}\times\col H_L(w^d)$.

		Conversely, for any $w\in\calB_L$, write $w = H_Lg$ and take
		$\nu = \delta_g$. Then,
		$m_1(\pushfwd{(H_L)}\delta_g) = [1;\,H_Lg] = [1;\,w]$, so every
		element of $\{1\}\times\col H_L(w^d)$ is realized.
	\end{proof}

	\begin{remark}[Degree-one moments versus the full measure]
		\label{rem:fundamental_interpretation}
		Corollary~\ref{cor:fundamental_bridge} extracts only first-order
		information from the measure-level Fundamental Lemma. The mean
		trajectory $\bar{w} = \int w\,d\mu$ retains the classical Hankel
		column-space structure but discards all higher-order distributional
		information such as variances and correlations across time steps.
		The full measure-level result
		(Theorem~\ref{thm:measure_level_FL}) preserves this richer structure.
	\end{remark}

	For controllable LTI systems, the measure-level Fundamental Lemma has an
	immediate operational consequence: any optimization over trajectory
	distributions reduces to an equivalent optimization over coefficient-space
	distributions, requiring no identified state-space model once the standard
	persistency-of-excitation condition is satisfied.

	\begin{corollary}[Data-driven optimization over behavioral measures]
		\label{cor:data_driven_opt}
		Under the assumptions of Theorem~\ref{thm:measure_level_FL}, for any Borel
		measurable path cost $c:\calW_L\to(-\infty,+\infty]$ that is bounded
		below, the optimization over behavioral measures is equivalent to an
		optimization over coefficient-space distributions:
		\begin{equation}
			\label{eq:data_driven_opt}
\ifdim\linewidth>350pt
\def\datadrivenoptdisplaybreak{}
\else
\def\datadrivenoptdisplaybreak{\\&\quad}
\fi
\begin{aligned}
&\inf_{\mu\in\mathcal{M}_{\mathcal{B}}^L}
			\int_{\calB_L} c(w)\,d\mu(w)\datadrivenoptdisplaybreak
			=
			\inf_{\nu\in\calP(\R^{N-L+1})}
			\int_{\R^{N-L+1}} c\bigl(H_L(w^d)\,g\bigr)\,d\nu(g).
\end{aligned}
		\end{equation}
	\end{corollary}

	\begin{proof}
		By Theorem~\ref{thm:measure_level_FL}, the map
		$\nu\mapsto\pushfwd{(H_L)}\nu$ is a surjection from
		$\calP(\R^{N-L+1})$ onto $\mathcal{M}_{\mathcal{B}}^L$: every
		behavioral measure can be written as a pushforward of some
		coefficient-space distribution. For any such representation
		$\mu = \pushfwd{(H_L)}\nu$, the change-of-variables formula for
		pushforward measures, valid in $[\inf c,+\infty]$ for Borel $c$ bounded
		below, gives
		\[
		\int_{\calB_L} c(w)\,d\mu(w)
		= \int_{\R^{N-L+1}} c\bigl(H_L(w^d)\,g\bigr)\,d\nu(g).
		\]
		Since every $\mu\in\mathcal{M}_{\mathcal{B}}^L$ admits such a
		representation, the two infima range over the same set of cost
		values and therefore coincide.
	\end{proof}

	The change of variables also carries expectation constraints through the
	Hankel matrix, which yields an exact coefficient-space form of constrained
	distributional DeePC.

	\begin{corollary}[Constrained distributional DeePC]
		\label{cor:cdeepc}
		Under the assumptions of Theorem~\ref{thm:measure_level_FL}, let
		$c_0,c_1,\dots,c_m:\calW_L\to(-\infty,+\infty]$ be Borel measurable and
		bounded below, and let $b\in\R^m$. Then
\ifdim\linewidth>350pt
		\begin{multline}
			\label{eq:cdeepc}
			\inf\Bigl\{\textstyle\int c_0\,d\mu : \mu\in\mathcal{M}_{\mathcal{B}}^L,\;
			\int c_j\,d\mu\le b_j,\; j=1,\dots,m\Bigr\}\\
			=\inf\Bigl\{\textstyle\int c_0(H_L g)\,d\nu :
			\nu\in\calP(\R^{N-L+1}),\;
			\int c_j(H_L g)\,d\nu\le b_j,\; j=1,\dots,m\Bigr\}.
		\end{multline}
\else
		\begin{multline}
			\label{eq:cdeepc}
			\inf\Bigl\{\textstyle\int c_0\,d\mu : \mu\in\mathcal{M}_{\mathcal{B}}^L,\;
			\int c_j\,d\mu\le b_j,\; j=1,\dots,m\Bigr\}\\
			=\inf\Bigl\{\textstyle\int c_0(H_L g)\,d\nu :
			\nu\in\calP(\R^{N-L+1}),\\
			\int c_j(H_L g)\,d\nu\le b_j,\; j=1,\dots,m\Bigr\}.
		\end{multline}
\fi
		The pushforward $\nu\mapsto\pushfwd{(H_L)}\nu$ maps the coefficient-law
		feasible set onto the trajectory-law feasible set. The lift
		$\mu\mapsto\pushfwd{(H_L^\dagger)}\mu$ is an injective right inverse,
		selecting a feasible coefficient law supported on the row space of $H_L$.
	\end{corollary}

	\begin{proof}
Apply the change-of-variables identity in the preceding proof to each
$c_j$, including the objective $c_0$. Theorem~\ref{thm:measure_level_FL}
supplies the surjective pushforward and canonical right inverse. They
preserve all constraint and cost integrals, proving the stated feasible-set
correspondence and equality of infima.
\end{proof}

On a compact coefficient support, finite continuous costs and constraint
functions define a compact probability-law program. Applying
Corollary~\ref{cor:compact_path_duality} to that compact coefficient space
gives zero duality gap whenever
that program is feasible, and multiplier attainment under strict
feasibility. The change-of-variables correspondence above also holds for
Borel bounded-below costs; compactness alone does not provide this stronger
duality statement for arbitrary Borel data.

		\begin{remark}[Distributional DeePC as an exact lift of DeePC]
			\label{rem:deepc}
			Standard DeePC~\cite{Coulson2019,Berberich2021} solves data-driven
			optimal control by finding a single coefficient vector $g$ such that
			the trajectory $w = H_L(w^d)g$ minimizes a cost and satisfies
			constraints. This is the special case of
			Corollary~\ref{cor:data_driven_opt} in which $\nu$ is restricted to
			a Dirac mass $\delta_g$. Replacing $g$ by a probability distribution
			$\nu$ on coefficient space yields a \emph{distributional DeePC}
			formulation over trajectory distributions. Whenever the finite minimum
			of an unconstrained expected-cost problem is attained, a Dirac law
			attains that value. The distributional lift is most useful when random initial
			laws, ensemble or distributional constraints, risk objectives,
			moment/covariance requirements, or ambiguity sets are part of the
			specification.
	\end{remark}

	\begin{remark}[Higher-order moment transfer]
		\label{rem:higher_moments}
			The measure-level factorization $\mu = \pushfwd{(H_L)}\nu$
			transfers moment computations from trajectory space to coefficient
			space. If $\mu$ has finite second moment, then the canonical lift
			$\nu=(H_L^\dagger)_{\#}\mu$ has finite second moment because
				\ifdim\linewidth>350pt
\[
				\int \|g\|^2\,d\nu(g)
				=
				\int \|H_L^\dagger w\|^2\,d\mu(w) \le
				\|H_L^\dagger\|_{\mathrm{op}}^2
				\int \|w\|^2\,d\mu(w).
				\]
\else
\[
				\begin{aligned}
				\int \|g\|^2\,d\nu(g)
				&=
				\int \|H_L^\dagger w\|^2\,d\mu(w) \\
				&\le
				\|H_L^\dagger\|_{\mathrm{op}}^2
				\int \|w\|^2\,d\mu(w).
				\end{aligned}
				\]
\fi
			Whenever $\nu$ has finite second moments, the mean and
			covariance of the trajectory distribution satisfy
		\[
		\E_\mu[w]=H_L\,\E_\nu[g],
		\qquad
		\mathrm{Cov}_\mu[w]=H_L\,\mathrm{Cov}_\nu[g]\,H_L^\top.
		\]
		These identities extend to higher-order moments, when finite, through
		Kronecker powers of $H_L$. For variance-penalized objectives, covariance
		constraints, or conservative moment-based surrogates for chance
		constraints, they translate the relevant mean/covariance conditions to
		coefficient space.
	\end{remark}

The quadratic case admits an affine-feedback reduction within established
covariance control~\cite{PilipovskyTsiotras2023,MoustroufisTsiotras2026}.
Affine state- and disturbance-feedback parametrizations are classical tools
for constrained linear control~\cite{GoulartKerriganMaciejowski2006}. The
following regression argument establishes optimality over all joint laws
for this particular upper-variance problem.

\begin{proposition}[Affine feedback under upper variance constraints]
\label{prop:affine_covariance_control}
Suppose that the complete observed past $W_p$ and the future inputs $U_f$
determine the future outputs by $Y_f=F_pW_p+F_uU_f$. For an LTI
input/output past window of length $T_{\mathrm{ini}}$, assume
$T_{\mathrm{ini}}$ is at least the lag so that this determinacy holds.
Fix the entire law of
$W_p$, with mean $\bar w_p$ and covariance $C_p\succ0$. Over square-integrable
joint laws with this past marginal and a fixed reference $y_{\mathrm{ref}}$, minimize
\[
\E\bigl[(Y_f-y_{\mathrm{ref}})^\top Q(Y_f-y_{\mathrm{ref}})+U_f^\top RU_f\bigr],
\qquad Q\succeq0,\quad R\succ0,
\]
subject only to upper bounds on selected components of
$\operatorname{diag}\operatorname{Cov}(Y_f)$. If feasible, this problem has an
optimal affine input schedule
\begin{equation}
\label{eq:affine_covariance_feedback}
U_f=\bar u_f+K(W_p-\bar w_p),\qquad
K=\Sigma_{up}C_p^{-1},
\end{equation}
where $\Sigma_{up}=\operatorname{Cov}(U_f,W_p)$. Every optimal joint law has
this form almost surely, and hence has zero conditional input covariance
given the past. Its external trajectory law is extreme in the
\emph{fixed-past-law slice}, the set of admissible joint laws with this
prescribed $W_p$ marginal.
\end{proposition}

\begin{proof}
For any feasible joint law, set
$\widetilde U_f=\bar u_f+\Sigma_{up}C_p^{-1}(W_p-\bar w_p)$ and
$\epsilon=U_f-\widetilde U_f$. Linear regression gives
$\E\epsilon=0$ and $\operatorname{Cov}(\epsilon,W_p)=0$.
Writing $S_\epsilon=\operatorname{Cov}(\epsilon)\succeq0$ and
$\widetilde Y_f=F_pW_p+F_u\widetilde U_f$, we have
\[
\begin{aligned}
\operatorname{Cov}(U_f)
 &=KC_pK^\top+S_\epsilon,\\
\operatorname{Cov}(Y_f)
 &=\operatorname{Cov}(\widetilde Y_f)+F_uS_\epsilon F_u^\top.
\end{aligned}
\]
The replacement preserves the entire past law and all means, respects every
upper variance bound, and reduces the cost by
\[
\operatorname{tr}\bigl[(F_u^\top QF_u+R)S_\epsilon\bigr].
\]
This quantity is strictly positive unless $S_\epsilon=0$, since $R\succ0$.
Thus an optimum must satisfy $U_f=\widetilde U_f$ almost surely. Applying
the reduction to any feasible law shows that the feasible affine parameter
set is nonempty. It is closed, and the positive-definite matrices $R,C_p$ make its
finite-cost sublevel sets bounded, so the continuous objective attains its
minimum. A deterministic conditional input law is
extreme among joint laws with a fixed past marginal by the same
fixed-marginal argument as Proposition~\ref{prop:slice}. The injective map
$(w_p,u_f)\mapsto(w_p,u_f,F_pw_p+F_uu_f)$ preserves this property.
\end{proof}

The full past is observed before the complete future input schedule is chosen;
there is no future process or measurement noise in this statement. The gain
need not have a triangular structure because every past coordinate is already
available. This is a fixed-past-law slice, rather than the initial-state slice
of Proposition~\ref{prop:slice}. Extra constraints such as lower variance
bounds or pathwise input bounds are outside the proposition's scope.

For a prescribed Gaussian past $\mathcal N(\bar w_p,C_p)$, the affine optimum
induces a jointly Gaussian, possibly degenerate trajectory law. Moreover,
every feasible joint mean/covariance pair has a Gaussian realization with
exactly that Gaussian past marginal. Consequently the moment semidefinite program (SDP) has the same
optimal value as the problem prescribing the entire Gaussian past law.
Non-Gaussian laws that merely match the past moments need not satisfy that
law constraint. For a non-Gaussian prescribed past, the affine-optimality
conclusion still holds, but the induced trajectory law need not be Gaussian.

	The Fundamental-Lemma factorization is exact only for noise-free data.
	An orthogonal projection onto an estimated behavioral subspace gives an
	explicit approximation of its trajectory laws.

	\begin{lemma}[Subspace-projection robustness]
		\label{lem:robust}
		Let $B$ and $\widehat B$ be linear subspaces of $\calW_L$ with orthogonal
		projectors $P$ and $\widehat P$, and let $\gamma:=\|P-\widehat P\|_2$.
		For $\mu\in\calP(B)$ with finite second moment, set
		$\widehat\mu:=\pushfwd{\widehat P}\mu\in\calP(\widehat B)$. Then
		\[
		W_2(\mu,\widehat\mu)^2
		=\int\|(I-\widehat P)w\|^2\,d\mu,
		\]
		and in particular
		\begin{equation}
			\label{eq:robust_w2}
			W_2(\mu,\widehat\mu)\le\gamma
			\Bigl(\int\|w\|^2\,d\mu\Bigr)^{1/2}.
		\end{equation}
		Here $W_2$ is quadratic Wasserstein distance~\cite{Villani2009}.
		For $M\ge0$, write
$\calP_M(B):=\{\mu\in\calP(B):\int\|w\|^2d\mu\le M^2\}$.
		The Hausdorff $W_2$-distance between $\calP_M(B)$ and
		$\calP_M(\widehat B)$ is exactly
\begin{equation}
\label{eq:robust_hausdorff}
d_H^{W_2}\bigl(\calP_M(B),\calP_M(\widehat B)\bigr)=\gamma M.
\end{equation}
	\end{lemma}

	\begin{proof}
		Every coupling of $\mu$ with a law supported on $\widehat B$ has squared
		transport cost at least $\int\operatorname{dist}(w,\widehat B)^2d\mu$.
		The coupling $(\id,\widehat P)_\#\mu$ attains this value, proving the
		equality. For $w=Pw$, the identity
		$(I-\widehat P)P=(I-\widehat P)(P-\widehat P)$ gives
		$\|(I-\widehat P)w\|\le\gamma\|w\|$. Integration proves
		\eqref{eq:robust_w2}. Orthogonal projection does not increase the second
		moment. The two directed Hausdorff distances have upper bounds
$M\|(I-\widehat P)P\|_2$ and $M\|(I-P)\widehat P\|_2$.
Each positive bound is attained by a Dirac law at $Mv$, where $v$ is a
maximizing unit vector in the corresponding source subspace: any target
law costs at least the squared distance of $Mv$ to the target subspace,
and its projected atom attains that distance within the moment budget.
Zero bounds are immediate. Finally,
\[
(P-\widehat P)^2=P(I-\widehat P)P+(I-P)\widehat P(I-P)
\]
is block diagonal relative to $\operatorname{ran}P\oplus\ker P$, with
block norms equal to the squares of the two directed projector norms.
Thus the maximum of the two directed projector norms equals $\gamma$, proving
\eqref{eq:robust_hausdorff}, including unequal ranks and $M=0$.
	\end{proof}

	\begin{corollary}[Noisy-Hankel law approximation]
		\label{cor:noisy_hankel}
		Assume Theorem~\ref{thm:measure_level_FL}, let $L$ be at least the lag,
		and suppose the minimal order $n$ is known. Set $n_B=L n_u+n$,
		$H=H_L(w^d)$ and $B=\col H$, so $\operatorname{rank}H=n_B$.
		For $\widetilde H=H+E$, let $\widehat B$ be the span of its leading $n_B$
		left singular vectors. Use the projectors and $\gamma$ of
		Lemma~\ref{lem:robust}. Write
		$s_{n_B}=\sigma_{n_B}(H)>0$, $\widehat s_{n_B}=\sigma_{n_B}(\widetilde H)$ and
		$\varepsilon=\|E\|_2$. If $\widehat s_{n_B}>0$, then
		\begin{equation}
		\label{eq:noisy_hankel_angle}
		\gamma\le\min\{1,\varepsilon/\widehat s_{n_B}\}.
		\end{equation}
		If $\varepsilon<s_{n_B}$, then
$\gamma\le\min\{1,\varepsilon/(s_{n_B}-\varepsilon)\}$.
The stronger condition $\varepsilon<s_{n_B}/2$ ensures separation of the
leading subspace; for $s_{n_B}/2\le\varepsilon<s_{n_B}$ the minimum with one is
the trivial bound and does not require separation.

		More specifically, suppose inputs are measured exactly and
		$\widetilde y_t^d=y_t^d+e_t$, where the $e_t$ are independent
		$\mathcal N(0,\sigma^2 I_{n_y})$ errors, independent of the clean record. For $0<\delta<1$, set
		\begin{equation}
		\label{eq:noisy_hankel_envelope}
		\varepsilon_\delta=\sigma\sqrt L
		\left(\sqrt{N n_y}+\sqrt{2\log(1/\delta)}\right).
		\end{equation}
		Conditional on the clean record, with probability at least $1-\delta$,
		\begin{equation}
		\label{eq:noisy_hankel_law}
		\begin{gathered}
		d_H^{W_2}\bigl(\calP_M(B),\calP_M(\widehat B)\bigr)
		\le M\eta_\delta,\\
		\eta_\delta=
		\begin{cases}
		\min\{1,\varepsilon_\delta/(s_{n_B}-\varepsilon_\delta)\},
		&\varepsilon_\delta<s_{n_B},\\
		1,&\varepsilon_\delta\ge s_{n_B}.
		\end{cases}
		\end{gathered}
		\end{equation}
		On the same event, $M\min\{1,\varepsilon_\delta/\widehat s_{n_B}\}$ is also
		a valid bound whenever $\widehat s_{n_B}>0$.
	\end{corollary}

	\begin{proof}
		Choose an orthonormal basis $U_\perp$ of $B^\perp$. The leading singular
		triplets satisfy
		$U_\perp^\top\widehat U_{n_B}\widehat\Sigma_{n_B}
		=U_\perp^\top E\widehat V_{n_B}$, since $U_\perp^\top H=0$.
		Both subspaces have dimension $n_B$, so
		$\gamma=\|U_\perp^\top\widehat U_{n_B}\|_2$, the sine of their largest
		principal angle~\cite{Wedin1972,StewartSun1990}. This proves
		\eqref{eq:noisy_hankel_angle}. Singular-value perturbation gives
		$\widehat s_{n_B}\ge s_{n_B}-\varepsilon$, while every remaining singular value
		is at most $\varepsilon$. Thus $\varepsilon<s_{n_B}/2$ separates the
		leading $n_B$ singular values.

		Each scalar output-measurement error appears in at most $L$ entries of
$E$, hence
		$\|E\|_2\le\|E\|_F\le\sqrt L\|e_{0:N-1}\|$.
		For $\sigma>0$, the squared norm
$\|e_{0:N-1}\|^2/\sigma^2$ is chi-square with $Nn_y$ degrees of freedom.
The Gaussian squared-norm bound~\cite[Lemma~1]{LaurentMassart2000}
		implies $\Prob\{\|E\|_2\le\varepsilon_\delta\}\ge1-\delta$.
The zero-noise case is immediate.
		Combining this event with the preceding bounds and
		Lemma~\ref{lem:robust} proves~\eqref{eq:noisy_hankel_law}.
	\end{proof}

	\begin{remark}[Rank, constraints and imperfect data]
		\label{rem:wedin}
		The principal-angle identity requires equal dimensions. If the subspace
		dimensions differ, $\|P-\widehat P\|_2=1$; using the full column space of
		a noisy Hankel matrix can therefore give only the trivial bound. The
		corollary assumes the correct truncation rank; an order upper bound alone
		does not determine it. Overlapping Hankel columns share errors, so
		\eqref{eq:noisy_hankel_envelope} uses independent errors in the original
		record, not independent matrix entries. The resulting Frobenius bound is
		conservative and need not improve as the record grows.
		Low-rank approximation and regularization address imperfect data in
		different ways~\cite{DorflerCoulsonMarkovsky2023}. Here the certificate
		controls the complete moment-bounded law sets. Projection can change a
		prescribed past marginal and can violate a variance cap. For a linear
		observable $a^\top w$ and $d=W_2(\mu,\widehat\mu)$, however,
		\[
		\begin{aligned}
		|\E_\mu a^\top W-\E_{\widehat\mu}a^\top W|&\le\|a\|d,\\
		|\operatorname{sd}_\mu(a^\top W)
		-\operatorname{sd}_{\widehat\mu}(a^\top W)|&\le\|a\|d.
		\end{aligned}
		\]
		These follow from the $L^2$ triangle inequality under a transport
		coupling and contraction after centering. A cap $v$ therefore becomes
		the relaxed cap $(\sqrt v+\|a\|d)^2$ for the projected law. The bound
		alone does not compare independently optimized laws or guarantee their
		true-plant feasibility.
	\end{remark}

		\subsection{Stochastic Extension}
		\label{sec:extensions}

		We now outline how the behavioral-measure framework extends to stochastic
		dynamics, where the next state is drawn from a transition kernel rather
		than determined by a map. Throughout this subsection the output map $h$ is
		noiseless and the transition kernels $\mathcal{K}_t$ are assumed known.

		Consider controlled stochastic dynamics
		$X_{t+1}\sim\mathcal{K}_t(\cdot\mid X_t,U_t)$, where each
		$\mathcal{K}_t$ is a Feller transition kernel from
		$\calX\times\calU$ to $\calX$: for every $\varphi\in C_b(\calX)$,
the function
$\mathcal K_t\varphi(x,u):=\int\varphi(\xi)\mathcal K_t(d\xi\mid x,u)$
belongs to $C_b(\calX\times\calU)$. Thus weak continuity is joint in state
and input.
		For each $t$, define the available history before applying the current
		control by
		\[
		H_t:=(X_0,U_0,Y_0,\ldots,X_{t-1},U_{t-1},Y_{t-1},X_t),
		\]
		with the convention $H_0=X_0$, and let $\mathsf H_t$ denote the
		corresponding history space. The current control $U_t$ is conditioned on
		separately, so the relevant conditioning variable is $(H_t,U_t)$.

		\begin{definition}[Stochastic behavioral-measure set]
			\label{def:stochastic_behavioral_measure}
			For a prescribed initial law $\rho_0\in\calP(\calX)$, let
			$\mathcal{M}_{\mathcal{B}}^{\mathrm{st}}(\rho_0)$ denote the set of all
			$\mu\in\calP(\Omega_T)$ such that $(X_0)_{\#}\mu=\rho_0$,
			$Y_t=h(X_t,U_t)$ $\mu$-almost surely for every $t=0,\dots,T-1$, and
			for every $t=0,\dots,T-1$, every
			$\psi\in C_b(\mathsf H_t\times\calU)$, and every
			$\varphi\in C_b(\calX)$,
			\begin{equation}
				\label{eq:stochastic_kernel_consistency}
				\int_{\Omega_T}\psi(H_t,U_t)\!\left[
				\varphi(X_{t+1})
				-\int_{\calX}\varphi(\xi)\,\mathcal{K}_t(d\xi\mid X_t,U_t)
				\right] d\mu = 0.
			\end{equation}
			On Polish (hence standard Borel) spaces this weak identity is
			equivalent to
			\[
			\mu\bigl(X_{t+1}\in\cdot\mid H_t,U_t\bigr)
			=
			\mathcal{K}_t(\cdot\mid X_t,U_t)
			\quad \mu\text{-a.s.}
			\]
			Indeed, products of bounded continuous test functions determine finite
			Borel measures on Polish spaces, and the usual monotone-class argument
			then identifies the regular conditional distribution.
		\end{definition}

		Condition~\eqref{eq:stochastic_kernel_consistency} is the stochastic
		analogue of graph support: it fixes the conditional transition kernel
		given the available history and current control, not merely the next-state
		marginal law. By contrast, the weaker identity
		\[
		\int_{\Omega_T}\varphi(X_{t+1})\,d\mu
		=
		\int_{\Omega_T}\!\int_{\calX}
		\varphi(\xi)\,\mathcal{K}_t(d\xi\mid X_t,U_t)\,d\mu
		\]
		for all $\varphi\in C_b(\calX)$ would only match the marginal law of
		$X_{t+1}$ and is therefore the stochastic analogue of the weak operator
		identities in Proposition~\ref{prop:operator_identities}; it defines only
		an outer consistency relaxation.
		Conditioning only on $(X_t,U_t)$ is also insufficient for a controlled
		Markov path law. For instance, take $T=2$, $\calX=\{0,1\}$, no control,
		and $\mathcal K_t(\{1\}\mid x)=1/2$ for every $x$. Let
		$X_0\sim\mathrm{Bernoulli}(1/2)$, let
		$X_1\sim\mathrm{Bernoulli}(1/2)$ be independent of $X_0$, and set
		$X_2=X_0$. Then $X_1\mid X_0$ and $X_2\mid X_1$ are both fair, but
		$\mu(X_2\in\cdot\mid X_0,X_1)=\delta_{X_0}$, violating the intended
		kernel.

		For $\mathcal K_t(\cdot\mid x,u)=\delta_{f(x,u)}$, the history-tested
		identity~\eqref{eq:stochastic_kernel_consistency} reduces to graph support
		by Remark~\ref{rem:operator_outer_approx}.

		\begin{proposition}[Convexity and weak closedness]
			\label{prop:stochastic_convex_closed}
			Assume that $h$ is continuous and that each transition kernel
			$\mathcal{K}_t$ is Feller. Then
			$\mathcal{M}_{\mathcal{B}}^{\mathrm{st}}(\rho_0)$ is convex and weakly
			closed in $\calP(\Omega_T)$.
		\end{proposition}

		\begin{proof}
			Convexity is immediate because the initial-law constraint, the output
			condition, and the identities
			\eqref{eq:stochastic_kernel_consistency} are all affine in $\mu$.
			For weak closedness, fix $t$, $\psi\in C_b(\mathsf H_t\times\calU)$, and
			$\varphi\in C_b(\calX)$, and define
				\ifdim\linewidth>350pt
\[
F_{t,\psi,\varphi}(\omega):=\psi(H_t(\omega),U_t(\omega))
\Bigl[\varphi(X_{t+1}(\omega))
-\int_{\calX}\varphi(\xi)\,\mathcal{K}_t(d\xi\mid X_t(\omega),U_t(\omega))\Bigr].
\]
\else
\begin{multline*}
				F_{t,\psi,\varphi}(\omega):=
				\psi(H_t(\omega),U_t(\omega))
				\left[
				\varphi(X_{t+1}(\omega)) \right.\\
				\left.
				-\int_{\calX}\varphi(\xi)\,\mathcal{K}_t(d\xi\mid X_t(\omega),U_t(\omega))
				\right].
				\end{multline*}
\fi
			The maps $H_t$, $U_t$, $X_t$, and $X_{t+1}$ are continuous coordinate
			projections on the finite product space, and the Feller property implies that
			$(x,u)\mapsto\int_{\calX}\varphi(\xi)\,\mathcal{K}_t(d\xi\mid x,u)$ is
			continuous. Hence $F_{t,\psi,\varphi}$ is bounded and continuous on
			$\Omega_T$. If $\mu^k\rightharpoonup\mu$ with each
			$\mu^k\in\mathcal{M}_{\mathcal{B}}^{\mathrm{st}}(\rho_0)$, then
			\[
			\int_{\Omega_T}F_{t,\psi,\varphi}\,d\mu
			=
			\lim_{k\to\infty}\int_{\Omega_T}F_{t,\psi,\varphi}\,d\mu^k
			=0,
			\]
			so \eqref{eq:stochastic_kernel_consistency} is preserved under weak
			limits. The initial-law constraint is preserved because $X_0$ is a
			continuous coordinate projection. For the output condition, let
			$\sigma_t^k=(X_t,U_t,Y_t)_{\#}\mu^k$ and
			$\sigma_t=(X_t,U_t,Y_t)_{\#}\mu$. Then
			$\sigma_t^k\rightharpoonup\sigma_t$, and the graph
			$\Gamma_t^y=\{(x,u,y):y=h(x,u)\}$ is closed by continuity of $h$.
			Since $\sigma_t^k(\Gamma_t^y)=1$, Portmanteau gives
			$\sigma_t(\Gamma_t^y)=1$. Therefore
			$\mathcal{M}_{\mathcal{B}}^{\mathrm{st}}(\rho_0)$ is weakly closed.
		\end{proof}

		The resulting occupation formulation recovers the standard
		controlled-Markov-process flow constraints
		\cite{HernandezLermaLasserre1996,Altman1999}, and the paper does not pursue the
		stochastic optimal-control theory beyond this structural extension. A
		different route was taken by Pan, Ou, and
		Faulwasser~\cite{PanOuFaulwasser2023}, surveyed
		in~\cite{Faulwasser2023Behavioral}: polynomial chaos expansions relate
		the dynamics of square-integrable random variables to expansion-coefficient
		behaviors and yield a stochastic Fundamental Lemma for linear systems.
		Their representation retains random-variable information beyond a finite
		set of moments. Here the stochastic object is a family of full trajectory
		laws constrained by prescribed history-conditional kernels. Extremal and
		duality questions have classical antecedents in strategic-measure and
		controlled-Markov-process theory~\cite{BhattBorkar1996,KurtzStockbridge1998};
		their detailed treatment is beyond the present scope.

		\section{Numerical Studies}
\label{sec:numerics}

The studies examine scalar polynomial graphs, nonlinear path control,
Hankel validation, covariance-constrained DeePC, and noisy-data approximation.
The nonlinear example imposes a constraint between nonadjacent times. The
final study tests law-space error bounds and evaluates policies designed
from noisy data. All reported numbers and figures come from the
JSON-to-LaTeX pipeline.\footnote{The numerical code, saved results, pinned
dependencies, and verification logs are supplied in the accompanying
\texttt{reproducibility.zip}, included as an ancillary file with the
arXiv source distribution. The archive contains its own file-hash manifest
and instructions for regenerating the experiments.}

\subsection{Scalar polynomial feasible-set structure}

Consider the constrained scalar polynomial graph
\begin{equation}
    \label{eq:poly_example}
    x_{t+1}=x_t^2+u_t,\quad
    (x_t,u_t,x_{t+1})\in[-\beta,\beta]^3,\quad \beta=\nExpOneBoxBound.
\end{equation}
Only pairs with $x_t^2+u_t\in[-\beta,\beta]$ are admissible. With
$m^{(t)}_{i,j,k}:=\E[x_t^i u_t^j x_{t+1}^k]$, the first-moment identity is
\begin{equation}
    \label{eq:moment_constraint}
    m^{(t)}_{0,0,1}=m^{(t)}_{2,0,0}+m^{(t)}_{0,1,0}.
\end{equation}
The mean next state depends on the current second moment, so the current
mean and mean input alone do not close the evolution equation. Keeping the
second moment as a separate coordinate makes this relation linear in the
moment entries.
The order-$r$ graph-ideal relaxation imposes every multiplied identity
\[
    m^{(t)}_{i,j,k+1}-m^{(t)}_{i+2,j,k}-m^{(t)}_{i,j+1,k}=0,
    \qquad i+j+k\le 2r-2,
\]
for nonnegative $i,j,k$, rather than matching only the marginal moments of
$x_{t+1}$ and $x_t^2+u_t$. Collecting all moments through degree $2r$ gives
\ifdim\linewidth>350pt
\begin{equation}
    \label{eq:lasserre_hierarchy}
    \mathcal R_r:=\{\bm m:\;
        A_{\mathrm{dyn}}\bm m=0,\quad m_{0,0,0}=1,\quad M_r(\bm m)\succeq0,\quad
        M_{r-d_j}(g_j\bm m)\succeq0\},
\end{equation}
\else
\begin{equation}
    \label{eq:lasserre_hierarchy}
    \begin{aligned}
    \mathcal R_r:=\{\bm m:\;&
        A_{\mathrm{dyn}}\bm m=0,\quad m_{0,0,0}=1,\\
        &M_r(\bm m)\succeq0,\quad
        M_{r-d_j}(g_j\bm m)\succeq0\},
    \end{aligned}
\end{equation}
\fi
where $g_j$ are $\beta^2-x_t^2$, $\beta^2-u_t^2$, and $\beta^2-x_{t+1}^2$, and
$d_j=\lceil\deg(g_j)/2\rceil$ \cite{Lasserre2008,Lasserre2010}.
The implementation retains all graph identities and removes their forced
null directions from the PSD blocks before solving.

We trace the projection onto $(\E[x_tu_t],\E[x_{t+1}])$ at
$\nExpOneNangles$ equally spaced support directions. The resulting polygon
areas are $\nExpOneAreaOne$, $\nExpOneAreaTwo$, and $\nExpOneAreaThree$ for
$r=\nExpOneOrderOne,\nExpOneOrderTwo,\nExpOneOrderThree$, respectively.
The latter two agree to the displayed precision. The grid-reference polygon
has area $\nExpOneReferenceArea$. These are finite-direction area estimates,
not exact set areas. We use CLARABEL
\nExpOneSolverVersion{} with feasibility and gap tolerances
$\nExpOneSolverTol$, retrying directions that fail its optimal-status check
with SCS \nExpOneFallbackVersion{} at tolerance $\nExpOneFallbackTol$.
Every accepted solve has optimal status. The largest graph-identity residual
is $\nExpOneGraphResidual$.

\begin{figure}[t]
    \centering
    \includegraphics[width=.72\linewidth]{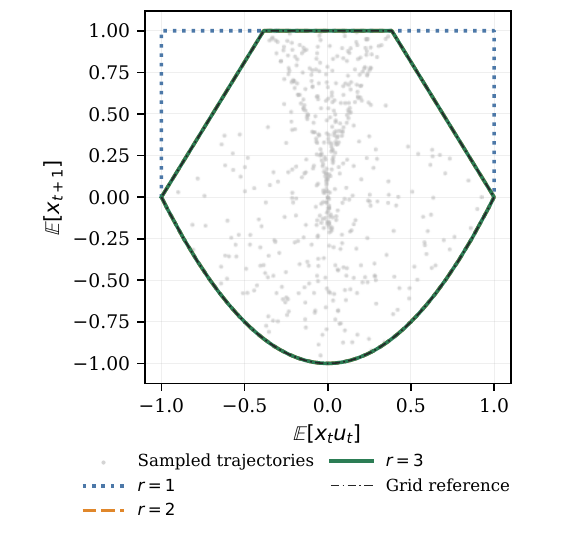}
    \caption{Scalar graph-ideal relaxations projected onto
    $(\E[x_tu_t],\E[x_{t+1}])$. Dotted, dashed, and solid curves correspond to
    $r=\nExpOneOrderOne,\nExpOneOrderTwo,\nExpOneOrderThree$;
    the latter two coincide at this resolution.
    The gray points are $\nExpOneNsamples$ accepted samples from independent
    uniform proposals for $(x_t,u_t)$ in $[-\beta,\beta]^2$, retaining only pairs with
    $x_t^2+u_t\in[-\beta,\beta]$.
    The black curve is a convex-hull reference from a
    $\nExpOneReferenceGrid\times\nExpOneReferenceGrid$ grid, not an exact boundary.}
    \label{fig:sos}
\end{figure}

\subsection{Nonlinear control with a nonadjacent path constraint}
\label{subsec:path_experiment}

Consider the scalar nonlinear system
\begin{equation}
\label{eq:path_example_dynamics}
x_{t+1}=\frac{x_t^2+u_t}{\nExpFiveDynamicsDivisor},
\qquad x_t,u_t\in[\nExpFiveSignalLo,\nExpFiveSignalHi],
\end{equation}
with horizon $T=\nExpFiveHorizon$, initial state
$X_0=\nExpFiveInitialState$, and reset
$X_{\nExpFiveResetTime}=\nExpFiveReset$ almost surely. We minimize
\ifdim\linewidth>350pt
\begin{equation}
\label{eq:path_example_control}
J(\mu)=\E_\mu\sum_{t=0}^{T-1}U_t^2,\qquad
\E_\mu X_{\nExpFiveFirstTime}=\E_\mu X_{\nExpFiveLastTime}=0,\qquad
\E_\mu[X_{\nExpFiveFirstTime}X_{\nExpFiveLastTime}]\le-\kappa,
\end{equation}
\else
\begin{equation}
\label{eq:path_example_control}
\begin{aligned}
J(\mu)&=\E_\mu\sum_{t=0}^{T-1}U_t^2,\\
\E_\mu X_{\nExpFiveFirstTime}&=\E_\mu X_{\nExpFiveLastTime}=0,\qquad
\E_\mu[X_{\nExpFiveFirstTime}X_{\nExpFiveLastTime}]\le-\kappa,
\end{aligned}
\end{equation}
\fi
where $0<\kappa\le\nExpFiveEndpointSquareBound$. The reset is equivalently
$\E[X_{\nExpFiveResetTime}^2]\le0$; the mean equalities can be paired
inequalities. Thus Theorem~\ref{thm:path_duality} applies, including its
feasibility-only zero-gap conclusion. Those equalities preclude strict
Slater feasibility. We compute on the exact compact reset face, using the
compact-path duality of Corollary~\ref{cor:compact_path_duality}, with the
initial condition, dynamics, reset, and signal bounds included in $K$.

Write $a=X_{\nExpFiveFirstTime}$, $b=X_{\nExpFiveLastTime}$, and
$c=U_{\nExpFiveFirstTime}$. The graph constraints give
$U_0=\nExpFiveControlScale a$, $c=-a^2$, and
$U_{\nExpFiveResetTime}=\nExpFiveControlScale b$, with
$|a|,|b|\le\nExpFiveEndpointBound$.
For any $\nExpFiveEnergyCoeff\kappa^2\le z\le\nExpFiveEndpointSquareBound$,
a feasible law takes equally likely signs $S=\pm1$ and sets
\begin{equation}
\label{eq:path_example_policy}
U_0=\nExpFiveControlScale S\sqrt z,\quad
U_{\nExpFiveFirstTime}=-z,\quad
U_{\nExpFiveResetTime}=-\frac{\nExpFiveControlScale S\kappa}{\sqrt z}.
\end{equation}
To display the two trajectory laws explicitly, let $\omega_z(s,t)$ be the
admissible path with state and input sequences
\[
\ifdim\linewidth>350pt
\def\resetpathdisplaybreak{\qquad}
\else
\def\resetpathdisplaybreak{\\&}
\fi
\begin{aligned}
&x=(0,s\sqrt z,0,t\kappa/\sqrt z),\resetpathdisplaybreak
u=(2s\sqrt z,-z,2t\kappa/\sqrt z),\qquad s,t\in\{-1,1\}.
\end{aligned}
\]
The retained-sign and independent-sign laws are
\begin{equation}
\label{eq:reset_coupling_laws}
\ifdim\linewidth>350pt
\def\resetlawdisplaybreak{\qquad}
\else
\def\resetlawdisplaybreak{\\&}
\fi
\begin{aligned}
&\mu_{\mathrm{mem}}=\frac12\sum_{s=\pm1}\delta_{\omega_z(s,-s)},\resetlawdisplaybreak
\mu_{\mathrm{ind}}=\frac14\sum_{s,t=\pm1}\delta_{\omega_z(s,t)}.
\end{aligned}
\end{equation}
The controller generating $\mu_{\mathrm{mem}}$ retains its sampled sign
through the reset. The law $\mu_{\mathrm{ind}}$ is the canonical
Markov-policy reconstruction: it has exactly the same occupation marginals,
terminal law, and additive cost, but its cross moment is
$\nExpFiveMarkovCross$ instead of $-\kappa$. It violates every positive
$\kappa$ constraint. A standard Markov-kernel policy also loses this
correlation when the state resets. Retaining the earlier state
as an augmented coordinate is an alternative exact occupation formulation.
For this fixed occupation tuple, Proposition~\ref{prop:occupation_fiber}
gives the entire cross-moment interval $[-\kappa,\kappa]$: the two endpoints
are obtained by $t=-s$ and $t=s$, while the independent-sign law gives zero.
At the saturated amplitudes $\sqrt z=\kappa/\sqrt z=1/2$, the interval is
$[-1/4,1/4]$.
The zero endpoint means also rule out a Dirac feasible law, linking this
example to the role of mixtures in Remark~\ref{rem:randomization}.

This benchmark has an independent global solution. For any feasible law,
let $z=\E[a^2]$. Jensen's and Cauchy--Schwarz inequalities imply
\[
\E[\nExpFiveEnergyCoeff a^2+a^4+\nExpFiveEnergyCoeff b^2]
\ge \nExpFiveEnergyCoeff z+z^2+
\frac{\nExpFiveEnergyCoeff\kappa^2}{z}.
\]
The endpoint bounds require
$\nExpFiveEnergyCoeff\kappa^2\le z\le\nExpFiveEndpointSquareBound$.
The symmetric law~\eqref{eq:path_example_policy} attains this bound for
every such $z$, so the global value over all trajectory laws is
\begin{equation}
\label{eq:path_example_value}
J^\star(\kappa)=
\min_{\nExpFiveEnergyCoeff\kappa^2\le z\le\nExpFiveEndpointSquareBound}
\left(\nExpFiveEnergyCoeff z+z^2+
\frac{\nExpFiveEnergyCoeff\kappa^2}{z}\right).
\end{equation}
Stationarity of~\eqref{eq:path_example_value} on its interior branch gives
$z^2+z^3/2=\kappa^2$. With $\eta=8\kappa/z$, this yields the
polynomial certificate, with $C(a,b)=4a^2+a^4+4b^2$,
\[
C(a,b)+\eta ab+z^2
=(a^2-z)^2+4\left(b+\frac{\kappa a}{z}\right)^2.
\]
It identifies the reset-face dual value $\eta\kappa-z^2$. This does not
assert a finite multiplier for the reset inequality on the unrestricted
path space. The scalar minimization~\eqref{eq:path_example_value} also
covers the branch where an endpoint bound is active; among the tested
levels, that branch occurs only at $\kappa=\nExpFiveLastCap$.

\begin{proposition}[Nonattainment of the unrestricted reset dual]
\label{prop:reset_dual_nonattainment}
For every $0<\kappa\le1/4$, the unrestricted-path dual of
\eqref{eq:path_example_control}, with the reset represented by
$\E[X_2^2]\le0$ and the means by paired inequalities, has zero gap but no
finite maximizing multiplier vector.
\end{proposition}

\begin{proof}
Zero gap follows from Theorem~\ref{thm:path_duality}. Let $z>0$ solve
\eqref{eq:path_example_value} and set $b_+=\kappa/\sqrt z>0$. In the
optimal law $\mu_{\mathrm{mem}}$, perturb only the atom with
$X_1=-\sqrt z$ and $X_3=b_+$ by setting
\[
X_2=\epsilon,\qquad U_1=-z+2\epsilon,\qquad
U_2=2b_+-\epsilon^2.
\]
For sufficiently small $\epsilon>0$, the dynamics and all signal bounds
hold, including when $b_+=1/2$. The endpoint means and cross moment remain
unchanged, while the expected energy changes by
$-2z\epsilon+2(1-b_+)\epsilon^2+\epsilon^4/2$ and the reset residual
becomes $\epsilon^2/2$. For any finite reset price $\nu\ge0$, the change
in the Lagrangian is
\[
-2z\epsilon+\bigl[2(1-b_+)+\nu/2\bigr]\epsilon^2+\epsilon^4/2<0
\]
for small enough $\epsilon$. Other constraint prices contribute no change.
Every finite-price inner infimum is therefore strictly below the primal
optimum. This is consistent with an attained finite-price certificate on
the exact reset face, whose minimization domain excludes the perturbation.
\end{proof}

We solve a moment relaxation at order $r=\nExpFiveOrder$ in $(a,c,b)$, retaining
$c+a^2=0$ with every monomial multiplier through degree $2r-2$ and all box
localizers. Linear reset and input identities are eliminated before
truncation. The full path graph, whose dimensions are reported for comparison rather
than as an additional solved formulation, has $\nExpFiveFullVariables$ free coordinates
before those substitutions, giving $\nExpFiveFullMoments$ moments and a
$\nExpFiveFullBlock\times\nExpFiveFullBlock$ moment block at this order.
The reduced graph uses $\nExpFiveGraphMoments$ moments and a
$\nExpFiveGraphBlock\times\nExpFiveGraphBlock$ block, compressed to
$\nExpFiveGraphCompressed\times\nExpFiveGraphCompressed$ along its known
ideal null direction, with $\nExpFiveGraphVariables$ localizers of size
$\nExpFiveGraphLocalBlock\times\nExpFiveGraphLocalBlock$.
Eliminating $c$ as well gives a separate quartic-cost relaxation with
$\nExpFiveElimMoments$ moments, a
$\nExpFiveElimBlock\times\nExpFiveElimBlock$ block, and
$\nExpFiveElimVariables$ localizers of size
$\nExpFiveElimLocalBlock\times\nExpFiveElimLocalBlock$.
These are distinct truncations; the analytic certificate validates their
observed accuracy. This small experiment tests temporal dependence and
does not establish computational scalability.

CLARABEL \nExpFiveSolverVersion{} returns optimal status for both
formulations at all $\nExpFiveCapCount$ caps in Table~\ref{tab:path_control},
using feasibility and gap tolerances $\nExpFiveTolerance$.
We project numerical dual Gram matrices onto the PSD cone. The remaining
polynomial residual in the reduced coordinates $\zeta$ has the form
$\Delta p(\zeta)=\sum_\alpha d_\alpha\zeta^\alpha$ and obeys
$|\Delta p(\zeta)|\le\sum_\alpha|d_\alpha|$ on the enclosing unit box.
Subtracting this coefficient bound from the dual objective gives floating-point
lower bounds, with largest correction $\nExpFiveDualResidual$;
the explicit law supplies upper bounds. The largest graph-equality residual
is $\nExpFiveGraphResidual$. At $\kappa=\nExpFiveCap$, the numerical interval
is $[\nExpFiveLower,\nExpFiveUpper]$, and the unrounded gap is
$\nExpFiveGap$. The analytic and SDP cap multipliers are
$\nExpFiveEta$ and $\nExpFiveSdpEta$. A separate
nonlinear programming (NLP) search with $\nExpFiveNlpStarts$ starting points uses sequential
least-squares programming within symmetric two-path laws and finds best
accepted feasible cost $\nExpFiveNlpCost$. It is a local-optimization
comparison; the polynomial argument supplies global verification.

\begin{table}[t]
\centering
\caption{Nonlinear path control. SDP lower and explicit-law upper bounds
are rounded outward to six decimals. NLP is the best accepted local-search value in the
symmetric two-path class.}
\label{tab:path_control}
\begin{tabular}{rrrr}
\hline
$\kappa$ & SDP lower & Law upper & Local NLP\\
\hline
\nExpFiveTableRows
\hline
\end{tabular}
\end{table}

\subsection{Data-driven LTI validation}

We validate Theorem~\ref{thm:measure_level_FL} on the single-input single-output
(SISO) system
\ifdim\linewidth>350pt
\begin{equation}
    \label{eq:lti_validation_system}
        x_{t+1}=\begin{bmatrix}\nExpThreeA\end{bmatrix}x_t
                    +\begin{bmatrix}\nExpThreeB\end{bmatrix}u_t,\qquad
        y_t=\begin{bmatrix}\nExpThreeC\end{bmatrix}x_t.
\end{equation}
\else
\begin{equation}
    \label{eq:lti_validation_system}
    \begin{aligned}
        x_{t+1}&=\begin{bmatrix}\nExpThreeA\end{bmatrix}x_t
                    +\begin{bmatrix}\nExpThreeB\end{bmatrix}u_t,\\
        y_t&=\begin{bmatrix}\nExpThreeC\end{bmatrix}x_t.
    \end{aligned}
\end{equation}
\fi
A noise-free trajectory of length $N=\nExpThreeN$ gives a depth-$\nExpThreeL$
Hankel matrix of size $\nExpThreeHankelRows\times\nExpThreeHankelCols$ and
rank $\nExpThreeRank$, matching the behavioral dimension $Ln_u+n_x$. The input Hankel
has full row rank at the required excitation order $\nExpThreePeOrder$.
For $\nExpThreeNval$ independent validation trajectories, the canonical
least-squares lifts $g_i=H_L^\dagger w_i$ give median and maximum residuals
$\nExpThreeMedRes$ and $\nExpThreeMaxRes$.
The mean of $\nExpThreeNmean$ trajectories has projection residual
$\nExpThreeDegOneRes$. The relative Frobenius residual in
\[
    \widehat{\mathrm{Cov}}[w]
        =H_L\widehat{\mathrm{Cov}}[g]H_L^\top
\]
is $\nExpThreeCovRes$, confirming the second-order identity of
Remark~\ref{rem:higher_moments} to machine precision.
Figure~\ref{fig:data_driven_validation} shows the validation residuals.

\begin{figure}[t]
    \centering
    \includegraphics[width=.72\linewidth]{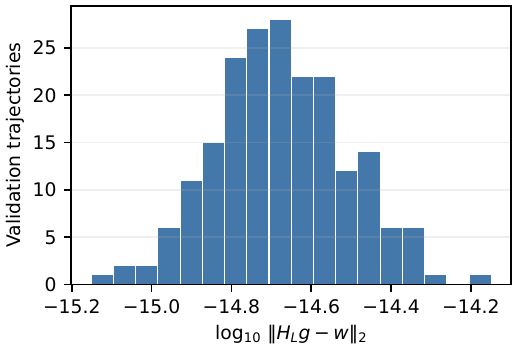}
    \caption{Hankel projection residuals for $\nExpThreeNval$ independent
    length-$\nExpThreeL$ validation trajectories of
    \eqref{eq:lti_validation_system}, using the noise-free dataset.}
    \label{fig:data_driven_validation}
\end{figure}

\begin{figure}[t]
    \centering
    \includegraphics[width=.95\textwidth]{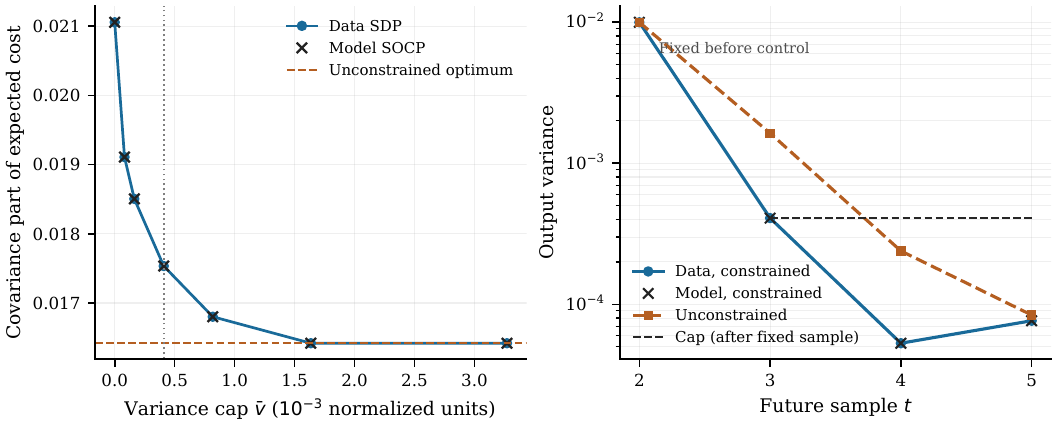}
    \caption{Matched covariance control. Left: covariance cost versus the
    cap on controllable outputs; markers for the data SDP and independently
    optimized model second-order cone program (SOCP) coincide. The dotted vertical line marks the headline
    cap. Right: output variances at that cap. The sample
    $t=\nExpFourTini$ is fixed before control and excluded from the cap,
    which applies from $\nExpFourFirstCapped$ through $\nExpFourLastStep$.}
    \label{fig:distributional_deepc}
\end{figure}

\subsection{Covariance Control with Matched Affine Feedback}
\label{subsec:covariance_control}

We use the same exact dataset and Hankel matrix as the preceding study.
The entire past $W_p=(U_0,Y_0,U_1,Y_1)$ is observed before selecting the
$\nExpFourFuture$ future inputs; there are no subsequent observations or
future disturbances. The observed window has length
$T_{\mathrm{ini}}=\nExpFourTini$ and reconstructs the control-start state.
To prescribe a feasible past law, choose independent centered Gaussian past
inputs and the two components $X_{1,T_{\mathrm{ini}}}$ and
$X_{2,T_{\mathrm{ini}}}$ of that state, with variances
$(\nExpFourPastInputVar,\nExpFourPastInputVar,
\nExpFourInitialOutputVar,\nExpFourInitialHiddenVar)$.
The primitive order is
$(U_0,U_1,X_{1,T_{\mathrm{ini}}},X_{2,T_{\mathrm{ini}}})$.
Backward propagation through the invertible plant determines the Gaussian
past law $\mathcal N(0,C_p)$, where $C_p\succ0$. State coordinates are
nondimensional, and the primary state ensemble is balanced.

Let $H_p,H_u,H_y$ be the past, future-input, and future-output row blocks
of $H_L$. A Gaussian coefficient law $\mathcal N(\bar g,\Sigma_g)$ induces
a generally degenerate Gaussian trajectory law with mean $H_L\bar g$ and
covariance $H_L\Sigma_gH_L^\top$. With reference
$y_{\mathrm{ref}}=\nExpFourReference$ and input weight $\rho_u=\nExpFourInputPenalty$, solve
\begin{equation}
\label{eq:deepc_gaussian_cost}
\ifdim\linewidth>350pt
\def\deepccostdisplaybreak{}
\else
\def\deepccostdisplaybreak{\\&\quad}
\fi
\begin{aligned}
\min_{\bar g,\Sigma_g}\quad
J&=\|H_y\bar g-y_{\mathrm{ref}}\mathbf1\|^2+\rho_u\|H_u\bar g\|^2\deepccostdisplaybreak
 +\operatorname{tr}(H_y\Sigma_gH_y^\top)
 +\rho_u\operatorname{tr}(H_u\Sigma_gH_u^\top).
\end{aligned}
\end{equation}
The equal input/output weights balance the two energy terms in these
normalized coordinates. Let $H_c$ contain only the output rows at
$\mathcal I=\{T_{\mathrm{ini}}+1,\ldots,L-1\}$. The constraints are
\begin{equation}
\label{eq:deepc_gaussian_constraints}
\ifdim\linewidth>350pt
\def\deepcconstraintbreak{\qquad}
\def\deepcconstraintalign{}
\else
\def\deepcconstraintbreak{\\}
\def\deepcconstraintalign{&}
\fi
\begin{aligned}
\Sigma_g\deepcconstraintalign\succeq0,\qquad H_p\bar g=0,\deepcconstraintbreak
H_p\Sigma_gH_p^\top\deepcconstraintalign=C_p,\deepcconstraintbreak
\operatorname{diag}(H_c\Sigma_gH_c^\top)\deepcconstraintalign\le\bar v\mathbf1.
\end{aligned}
\end{equation}
The first future output is already fixed by the observed past and has
variance $\nExpFourInitialOutputVar$, so it is reported separately from the
capped outputs. Variances are evaluated across the prescribed past ensemble.
Proposition~\ref{prop:affine_covariance_control} explains
why the Gaussian family attains the optimum for the prescribed Gaussian
past law. The mean problem separates from the covariance problem; once its
optimal mean is fixed, the variance caps are expectations of fixed squared
path functions, as in Corollary~\ref{cor:cdeepc}.

The data design solves a rank-reduced covariance SDP using only Hankel
blocks and the prescribed past law. The matched model benchmark
independently builds $Y_f=\mathcal O_fX_{T_{\mathrm{ini}}}+\mathcal T_fU_f$
from powers of $A,B,C$ and optimizes affine gains by a second-order cone
program. It uses the same past information, ensemble, horizon, cost and cap.
It does not use the data-optimized covariance. Both use CLARABEL
\nExpFourSolverVersion{} with feasibility and gap tolerances
$\nExpFourTolerance$. All reported solves return optimal status.

Denote the covariance part of~\eqref{eq:deepc_gaussian_cost} by
$J_{\mathrm{cov}}$. Every cap shares the nominal cost
$J_{\mathrm{mean}}=\nExpFourNominalCost$. Let $v_{\mathrm{free}}$ be the
largest variance among the capped outputs for the unconstrained optimum;
for the primary ensemble, $v_{\mathrm{free}}=\nExpFourCapScale$.
Table~\ref{tab:deepc_covariance} sweeps
$\bar v=\alpha v_{\mathrm{free}}$ from the exactly feasible zero cap to the
inactive regime. The zero-cap covariance cost exceeds the unconstrained
value by $\nExpFourCovIncreasePct\%$.
At the headline fraction $\alpha=\nExpFourCapFraction$,
$\bar v\simeq\nExpFourCap$, the data/model total costs are
$\nExpFourJcon$ and $\nExpFourJmodel$.

\begin{table}[t]
\centering
\caption{Matched covariance cost for the primary ensemble.
The cap is $\bar v=\alpha v_{\mathrm{free}}$.}
\label{tab:deepc_covariance}
\begin{tabular}{rrrr}
\hline
$\alpha$ & $\bar v$ & Data $J_{\mathrm{cov}}$ & Model $J_{\mathrm{cov}}$\\
\hline
\nExpFourSweepRows
\hline
\end{tabular}
\end{table}

The feedback is extracted using $\Sigma_{up}=H_u\Sigma_gH_p^\top$ in
\eqref{eq:affine_covariance_feedback}. Its columns follow the observed
past order $(u_0,y_0,u_1,y_1)$. At this cap it has rounded coefficients
\[
\begin{gathered}
\bar u_f^\top\simeq\nExpFourMeanInputsRow,\\
K\simeq\nExpFourGainPast.
\end{gathered}
\]
The last gain row is zero because the final input affects no output within
this no-direct-feedthrough prediction window and is penalized positively.
Full-precision coefficients are used for every feasibility check.
The conditional input covariance has operator norm
$\nExpFourConditionalNorm$. Removing its numerical residual changes the
expected cost by at most $\nExpFourAffineCostDiff$ and gives maximum cap
violation $\nExpFourAffineCapViolation$ over all tested cases.
These checks concern future inputs. We use the canonical row-space
coefficient lift; arbitrary additional noise in $\ker H_L$ does not affect
the trajectory.

We repeat the comparison for $\nExpFourSensitivityCount$ additional state
ensembles in Table~\ref{tab:covariance_sensitivity}, using the same relative
cap $\alpha=\nExpFourCapFraction$ for each ensemble. Across all matched
cases, the largest total-cost difference is $\nExpFourMaxCostGap$, and the
largest whitened-feedback gain difference is $\nExpFourMaxGainGap$.
An independent verifier propagates the reported laws through the plant,
checks affine realization and cap monotonicity, and evaluates separate
convex quadratic-program dual bounds. The comparison establishes recovery
of the model-based optimum under the exact-data assumptions.

\begin{table}[t]
\centering
\caption{State-variance sensitivity at the same relative cap.}
\label{tab:covariance_sensitivity}
\begin{tabular}{lrr}
\hline
State variances & Data $J_{\mathrm{cov}}$ & Model $J_{\mathrm{cov}}$\\
\hline
\nExpFourSensitivityRows
\hline
\end{tabular}
\end{table}

All uncertainty is in the initial ensemble. No pathwise state or input
bounds are imposed, as required by
Proposition~\ref{prop:affine_covariance_control}. The affine policy is designed
for the prescribed ensemble of past windows. After one window is observed,
the entire future input sequence is selected, with no later updates.

\subsection{Noisy-Data Law Approximation and Policy Evaluation}
\label{subsec:noisy_data}

We now add output measurement noise to the same length-$\nExpSixN$
training record, leaving the inputs exact. We use the known rank
$n_B=\nExpSixRank$ and window length $L=\nExpSixL$ in
Corollary~\ref{cor:noisy_hankel}. The noise standard deviations are
$\sigma\in\{\nExpSixSigmaList\}$. For each level we use
$\nExpSixReplicates$ independent records, with seeds
$\nExpSixNoiseSeed$ through $\nExpSixLastNoiseSeed$. Within each replicate,
the same standardized Gaussian record is scaled across levels. The observed
control past remains exact; there are no future disturbances or policy updates.

First fix the clean affine-optimal Gaussian law $\mu_\star$ at the headline
cap of Subsection~\ref{subsec:covariance_control}, with mean $m_\star$ and
covariance factor $F_\star$. For each noisy record, form the leading-rank
projector $\widehat P$ and the projected law
$\widehat\mu_\star=\widehat P_\#\mu_\star$. Lemma~\ref{lem:robust} gives
the exact distance without matrix square roots:
\begin{equation}
\label{eq:noisy_gaussian_distance}
\ifdim\linewidth>350pt
\def\noisydistancedisplaybreak{}
\else
\def\noisydistancedisplaybreak{\\&}
\fi
\begin{aligned}
W_2(\mu_\star,\widehat\mu_\star)^2
={}&\|(I-\widehat P)m_\star\|^2\noisydistancedisplaybreak
+\|(I-\widehat P)F_\star\|_F^2.
\end{aligned}
\end{equation}
The reference has second-moment radius
$M=(\int\|w\|^2\,d\mu_\star)^{1/2}=\nExpSixMomentRadius$,
and the clean Hankel singular value is $s_{n_B}=\nExpSixSmallestSingular$.
Figure~\ref{fig:noisy_data} compares this distance with $\gamma M$,
$M\min\{1,\varepsilon/\widehat s_{n_B}\}$ and $M\eta_\delta$ for
$\delta=\nExpSixDelta$. The last is a
$\nExpSixConfidence\%$ bound for one record, conditional on the clean data;
the displayed replicate ranges are empirical ranges, not confidence intervals.
The clean $s_{n_B}$ and realized $\varepsilon$ are available here for validation.
With a known noise level, the observed-denominator alternative in the
corollary does not require the clean singular value.
At $\sigma=\nExpSixHeadlineSigma$, the median distance is
$\nExpSixHeadlineWasserstein$, compared with projector, realized-noise and
noise-envelope bounds $\nExpSixHeadlineProjectionBound$,
$\nExpSixHeadlineHankelBound$ and $\nExpSixHeadlineEnvelopeBound$.
The conservative noise envelope reaches the trivial bound $M$ at
$\sigma=\nExpSixSigmaMax$, where $\varepsilon_\delta\ge s_{n_B}$ and
$\eta_\delta=1$.

\begin{figure}[t]
\centering
\includegraphics[width=\textwidth]{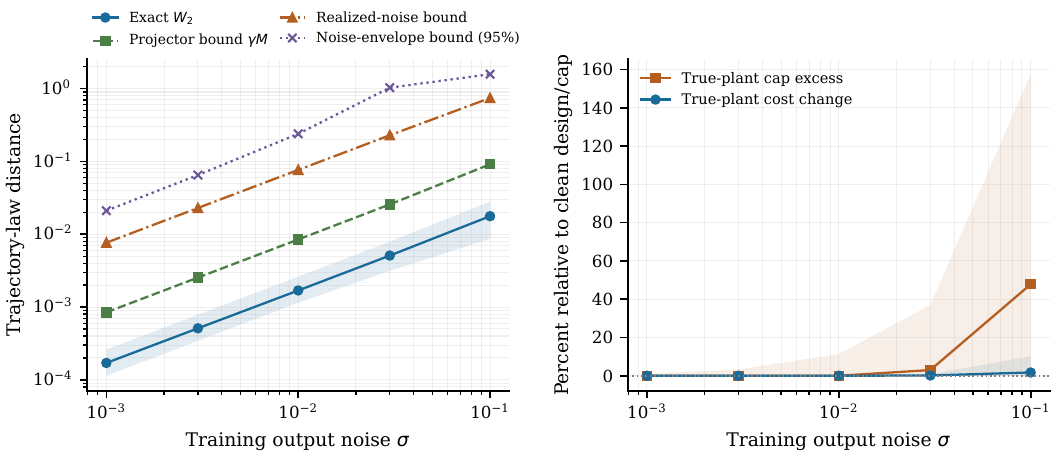}
\caption{Noisy training outputs, with exact observed control past. Left:
distance from the clean reference law to its projection, and the three
bounds described in the text. Right: true-plant performance of policies
independently optimized on the noisy subspaces. Cost change is relative to
the clean design; cap excess is relative to the prescribed variance cap.
Lines show medians and shaded bands show replicate minima and maxima;
the left band concerns the distance. The noise-envelope probability
applies to each record separately.}
\label{fig:noisy_data}
\end{figure}

Next, independently optimize an affine schedule on each noisy subspace,
with the same prescribed Gaussian past law, cost weights, reference and
variance cap as the clean design. We use the invertible coordinates
$(w_p,u_f)$ on the estimated subspace to construct its output prediction
map from the truncated singular value decomposition (SVD) alone.
The largest coordinate condition number
is $\nExpSixMaxCondition$. The affine reduction in
Proposition~\ref{prop:affine_covariance_control} gives a convex
second-order-cone program. All $\nExpSixTrialCount$ solves return optimal
status, and the largest surrogate cap violation is
$\nExpSixMaxSurrogateCapViolation$. We then apply the saved schedule to the
true plant, using the same exact past ensemble. This evaluation uses the
plant matrices; the surrogate optimization does not.

\begin{table*}[t]
\centering
\caption{Noisy-data sweep. The distance, projector bound, plant cost, and
first excess column are componentwise medians; the next column is the
largest excess among the replicates. Excess is the
nonnegative percentage by which the true plant's largest capped variance
exceeds $\bar v$. The last column counts policies whose absolute variance
violation exceeds $\nExpSixCapViolationTolerance$, out of
$\nExpSixReplicates$ records at that noise level.}
\label{tab:noisy_data}
\setlength{\tabcolsep}{4pt}
\begin{tabular}{rrrrrrr}
\hline
$\sigma$ & $W_2$ & $\gamma M$ & $J_{\mathrm{plant}}$
& \shortstack{Median excess\\(\%)} & \shortstack{Max excess\\(\%)} & Violations\\
\hline
\nExpSixSweepRowsDetailed
\hline
\end{tabular}
\end{table*}

Table~\ref{tab:noisy_data} reports the exact projection distance separately
from the executed controller's performance. At the highest noise level,
the median true-plant cap excess is $\nExpSixLastMedianCapExcessPct\%$;
the largest excess across all records is $\nExpSixMaxCapExcessPct\%$.
The largest true-plant cost increase over the clean design is
$\nExpSixMaxCostGapPct\%$. The surrogate can predict a feasible variance
while execution violates it. Moreover, the estimated subspace need not
satisfy exact LTI or causality identities, so its first future output need
not be fixed before control acts. We retain the same capped times as in the
clean design; the actual plant's first future output remains fixed.
Enforcing a causal LTI realization would require additional structural
constraints on the fitted model. Structured low-rank approximation and
regularized DeePC offer related treatments of imperfect
data~\cite{DorflerCoulsonMarkovsky2023}, with guarantees determined by their
respective model and optimization assumptions.

These are two distinct tests. Projection tests the law-approximation
bound and can change the prescribed past law. Reoptimization enforces the
past law within the surrogate but does not inherit a true-plant constraint
guarantee. The noisy-Hankel result certifies geometric approximation of
moment-bounded law sets; robust covariance-control formulations address
controller feasibility under data uncertainty using additional
analysis~\cite{PilipovskyTsiotras2024,PilipovskyTsiotras2026DUST}.
An independent verifier checks Gaussian transport by a
factor-based formula, the perturbation bounds and the saved policies by
direct simulation, and evaluates separate convex-program dual bounds.

\begin{remark}[Non-polynomial dynamics]
    \label{rem:nonpoly}
    The closed-behavior geometry of Lemma~\ref{lem:closed_behavior} needs no
    polynomial representation. The realization-based results use
    Assumption~\ref{ass:standing}. Polynomial dynamics are used here to obtain
    moment-SOS approximations. Alternatives include polynomial approximation
    followed by a moment relaxation, empirical trajectory approximations,
    and linear behavioral models in exact Koopman-invariant coordinates when
    such coordinates exist~\cite{LianWangJones2021Koopman}.
\end{remark}

\section{Conclusions}
\label{sec:conclusions}

The framework begins with laws on an arbitrary closed behavior, whose
probability-measure geometry is independent of a realization. State-space
coordinates then define initial-law slices and connect this description to
classical occupation-control duality. The full-path constrained theorem
handles continuous trajectory functions. The nonlinear reset experiment
shows that identical occupation marginals can conceal a correlation required
for feasibility, and gives matching analytic and numerical value bounds.
Occupation marginals preserve additive costs but do not determine the joint
law across time. Under the noise-free LTI assumptions, the classical
Fundamental Lemma yields an equivalent factorization of these laws in Hankel
coordinates. The constrained coefficient-law correspondence and the noisy-Hankel
projection bounds describe operational consequences of this formulation.
For quadratic covariance control, the affine reduction explains the
numerical recovery of the matched model-based optimum across the cap sweep.
Under output measurement noise and a known truncation rank, the law-set
bound quantifies geometric approximation. The noisy-data study separately
measures the true-plant cost and variance-cap violations of policies designed
on the surrogate. Projection alone preserves neither the prescribed past
law nor the controller's constrained feasibility.
True-plant guarantees for reoptimized controllers require a separate
robust-control or coordinate-alignment analysis; the clean reference law's
projection radius is not such a certificate.

The stochastic extension replaces graph support with history-conditional
kernel consistency. A data-driven factorization of the full trajectory laws
consistent with prescribed kernels remains open in this framework. Further
directions include noisy online observations, rank selection and constrained
controller guarantees, nonlinear representations, and scalable approximations
beyond low-dimensional moment relaxations.
For a moment vector in $d$ variables, the number of monomials through
degree $2r$ is $\binom{d+2r}{2r}$. The examples remain small, and the rapid
growth of the SDP blocks with dimension and order remains a computational
limitation.

\section*{Acknowledgment}
The author acknowledges the use of AI assistants for manuscript review,
revision, and preparation.

\bibliographystyle{IEEEtran}
	\bibliography{refs}

\end{document}